\documentclass[12pt]{article}
\usepackage{amsfonts}
\usepackage{amsmath}
\usepackage{amssymb}
\usepackage{amsthm}
\usepackage{enumerate}
\usepackage[hang,flushmargin]{footmisc}

\usepackage{graphicx}
\graphicspath{{eps/}}
\usepackage{epstopdf}
\usepackage{epsfig}
\DeclareMathOperator{\Tr}{Tr}
 
\usepackage{bm} 

\newtheorem{theorem}{Theorem}[section]
\newtheorem{proposition}{Proposition}[section]
\newtheorem{lemma}{Lemma}[section]
\newtheorem{corollary}{Corollary}[section]

\theoremstyle{definition}
\newtheorem{definition}{Definition}[section]

\begin{document}

\begin{center}
\vskip 1cm{\LARGE\bf{Upper Bounds for Stern's Diatomic Sequence and Related Sequences} 
\vskip 1cm
\large
Colin Defant\\
Department of Mathematics\\
University of Florida\\
United States\\
cdefant@ufl.edu}
\end{center}
\vskip .2 in

\begin{abstract} 
Let $(s_2(n))_{n=0}^\infty$ denote Stern's diatomic sequence. For $n\geq 2$, we may view $s_2(n)$ as the number of partitions of $n-1$ into powers of $2$ with each part occurring at most twice. More generally, for integers $b,n\geq 2$, let $s_b(n)$ denote the number of partitions of $n-1$ into powers of $b$ with each part occurring at most $b$ times. Using this combinatorial interpretation of the sequences $s_b(n)$, we use the transfer-matrix method to develop a means of calculating $s_b(n)$ for certain values of $n$. This then allows us to derive upper bounds for $s_b(n)$ for certain values of $n$. In the special case $b=2$, our bounds improve upon the current upper bounds for the Stern sequence. In addition, we are able to prove that $\displaystyle{\limsup_{n\rightarrow\infty}\frac{s_b(n)}{n^{\log_b\phi}}=\frac{(b^2-1)^{\log_b\phi}}{\sqrt 5}}$.
\end{abstract} 

\bigskip

\section{Introduction} 
Throughout this paper, $F_m$ and $L_m$ will denote the Fibonacci and the Lucas numbers. We have $F_{m+2}=F_{m+1}+F_m$ and $L_{m+2}=L_{m+1}+L_m$ for all integers $m$ (including negative integers). We convene to use the initial values $F_1=F_2=1$, $L_1=1$, and $L_2=3$. We also let $\phi=\dfrac{1+\sqrt 5}{2}$ and $\overline{\phi}=\dfrac{1-\sqrt5}{2}=\dfrac{-1}{\phi}$. The symbol $\mathbb N$ will denote the set of positive integers. 
\par 
Problem B1 of the 2014 William Lowell Putnam Competition defines a \emph{base $10$ over-expansion} of a positive integer $N$ to be an expression of the form \[N=d_k10^k+d_{k-1}10^{k-1}+\cdots+d_010^0\] with $d_k\neq 0$ and $d_i\in\{0,1,2,\ldots,10\}$ for all $i$. We may generalize (and slightly modify) this notion to obtain the following definition.
\begin{definition} \label{Def1}
Let $b\geq 2$ be an integer. A \emph{base $b$ over-expansion} of a positive integer $N$ is a word $d_kd_{k-1}\cdots d_0$ over the alphabet $\{0,1,\ldots,b\}$ such that $d_k\neq 0$ and $\displaystyle{\sum_{i=0}^k}d_ib^i=N$. We refer to the letter $d_i$ as the \emph{$i^{th}$ digit of the expansion}. It is well-known that each positive integer $N$ has a unique base $b$ over-expansion that does not contain the letter (or digit) $b$; we refer to this expansion as the \emph{ordinary base $b$ expansion} of $N$. 
\end{definition} 
\par
The \emph{Stern-Brocot sequence}, also known as \emph{Stern's diatomic sequence} or simply \emph{Stern's sequence}, is defined by the simple recurrence relations \[s(2n)=s(n)\hspace{.5cm}\text{and}\hspace{.5cm}s(2n+1)=s(n)+s(n+1)\] for all nonnegative integers $n$, where $s(0)=0$. This sequence has found numerous applications in number theory and combinatorics, and it has several interesting properties which relate it to the Fibonacci sequence. For $n\geq 2$, it is well-known that $s(n)$ is the number of base $2$ over-expansions (also known as hyperbinary expansions) of $n-1$ \cite{Carlitz64}. To generalize Stern's sequence, let $s_b(n)$ denote the number of base $b$ over-expansions of $n-1$. Equivalently, one may wish to think of $s_b(n)$ as the number of partitions of $n-1$ into powers of $b$ with each part occurring at most $b$ times. We convene to let $s_b(0)=0$ and $s_b(1)=1$. The sequence $s_b(n)$ satisfies the recurrence relations $s_b(bn)=s_b(n)$, $s_b(bn+1)=s_b(n)+s_b(n+1)$, and $s_b(bn+i)=s_b(n+1)$ for $i\in\{2,3,\ldots,b-1\}$. Equivalently,
\begin{equation} \label{Eq1}
s_b(n)=\begin{cases} s_b\left(\frac nb\right), & \mbox{if } b\equiv 0\pmod b; \\ s_b\left(\frac{n-1}{b}\right)+s_b\left(\frac{n-1}{b}+1\right), & \mbox{if } n\equiv 1\pmod b; \\ s_b\left(\frac{n-i}{b}+1\right), & \mbox{if } n\equiv i\pmod b\mbox{ and }2\leq i<b. \end{cases} 
\end{equation}   
Using \eqref{Eq1}, one may easily prove the following lemma. 
\begin{lemma} \label{Lem2}
Let $n$ be a positive integer. If $n\equiv 1\pmod{b^2}$, then \[s_b(n)=s_b\left(\frac{n-1}{b^2}\right)+s_b\left(\frac{n+b-1}{b}\right).\] If $n\equiv b+1\pmod{b^2}$, then \[s_b(n)=s_b\left(\frac{n-1}{b}\right)+s_b\left(\frac{n+b^2-b-1}{b^2}\right).\]
\end{lemma}
\par 
Calkin and Wilf \cite{Calkin09} determined that \[0.958854...=\frac{3^{\log_2\phi}}{\sqrt 5}\leq\limsup_{n\rightarrow\infty}\frac{s(n)}{n^{\log_2\phi}}\leq\frac{1+\phi}{2}=1.170820...,\] and they asked for the exact value of $\displaystyle{\limsup\frac{s(n)}{n^{\log_2\phi}}}$. Here, we give upper bounds for the values of $s_b(n)$ for any integer $b\geq 2$, from which we will deduce that \[\limsup_{n\rightarrow\infty}\frac{s_b(n)}{n^{\log_b\phi}}=\frac{(b^2-1)^{\log_b\phi}}{\sqrt 5}.\] In particular, $\displaystyle{\limsup_{n\rightarrow\infty}\frac{s(n)}{n^{\log_2\phi}}=\frac{3^{\log_2\phi}}{\sqrt 5}}$. While preparing this manuscript, the author discovered that Coons and Tyler \cite{Coons14} had already determined this value of the supremum limit of $\dfrac{s(n)}{n^{\log_2\phi}}$; in the same paper, they mention that this problem actually dates back to Berlekamp, Conway, and Guy in 1982 \cite{Berlekamp82}. However, we have no fear that our results are unoriginal because the bounds we derive apply to the more general family of sequences $s_b(n)$ and are stronger than those given in \cite{Coons14}. In addition, our methods of proof are quite different from those used in \cite{Coons14}. Coons and Tyler use clever analytic estimates to prove their results from the recursive definition of $s(n)$. By contrast, we will make heavy use of the interpretation of $s_b(n)$ as the number of base $b$ over-expansions of $n-1$ in order to prove several of our most important results. In particular, we will combine this combinatorial interpretation of $s_b(n)$ with the transfer-matrix method in order to prove Theorem \ref{Thm2}. In turn, Theorem \ref{Thm2} will allow us to determine the maximum values of $s_b(n)$ when $n$ is restricted to certain intervals. 
\section{Determining Maximum Values}   
Throughout this section, fix an integer $b\geq 2$. Our goal is to derive upper bounds for the numbers $s_b(n)$, particularly those values of $n$ that are slightly larger than a power of $b$ (we will make this precise soon). To do so, we will make use of the following sequence. 
\begin{definition} \label{Def2}
Let $h_1=h_2=1$. For $m\geq 3$, let \[h_m=1+\sum_{i=0}^{\left\lfloor\frac{m-3}{2}\right\rfloor}b^{m-2-2i}.\] 
\end{definition} 
Alternatively, we may calculate $h_m$ using the recurrence relation 
\begin{equation} \label{Eq2}  
h_m=\begin{cases} bh_{m-1}-b+1, & \mbox{if } 2\vert m; \\ bh_{m-1}+1, & \mbox{if } 2\nmid m \end{cases}
\end{equation} along with the initial value $h_1=1$. For example, $h_3=b+1$, $h_4=b^2+1$, $h_5=b^3+b+1$, and $h_6=b^4+b^2+1$. It is important to note that $h_m\equiv 1\pmod b$ for all $m\in\mathbb N$. 
We state the following lemma for easy reference, although we omit the proof because it is fairly straightforward. 
\begin{lemma} \label{Lem16} 
For any positive integer $m$, \[h_{m+1}-h_m=\frac{b^m+(-1)^mb}{b+1}.\]
\end{lemma}
The following lemma lists three simple but useful observations about the numbers $s_b(n)$. We omit the proof because it follows easily from \eqref{Eq1}. 
\begin{lemma} \label{Lem1}
Let $n$ and $k$ be nonnegative integers with $b^k\leq n\leq b^{k+1}$. 
\begin{enumerate}[i.] 
\item If $n=jb^k$ for some $j\in\{1,2,\ldots,b-1\}$, then $s_b(n)=1$. 
\item If $k=1$, then $s_b(n)\leq 2$, where equality holds if and only if $n\equiv 1\pmod b$. 
\item If $k\geq 1$ and $n\not\equiv 1\pmod b$, then $s_b(n)=s_b(n')$ for some integer $n'$ with $b^{k-1}\leq n'\leq b^k$. 
\end{enumerate}
\end{lemma}   
\begin{proposition}\label{Prop1} 
Let $k$ and $n$ be nonnegative integers with $b^k\leq n<b^{k+1}$. We may write $n$ uniquely in the form $n=jb^k+t$, where $j\in\{1,2,\ldots,b-1\}$ and $t\in\{0,1,\ldots,b^k-1\}$. We have $s_b(n)\leq F_{k+2}$, where equality holds if and only if $t=h_k$ or $t=h_{k+1}$. 
\end{proposition} 
\begin{proof} 
That we may write $n$ uniquely in the form $n=jb^k+t$ is trivial. The proof of the rest of the proposition is by induction on $k$. The case $k=0$ is immediate from the first part of Lemma \ref{Lem1}. The case $k=1$ follows from the second part of the same lemma.
Now, assume $k>1$. We divide the proof into three cases. 
\par 
\noindent Case 1: In this case, suppose $t=h_k$ or $t=h_{k+1}$. We will assume that $t=h_k$ and that $k$ is even; a similar argument holds if $k$ is odd or $t=h_{k+1}$. Thus, $n=jb^k+h_k$. Because $k$ is even, we may use \eqref{Eq2} to write $h_k=bh_{k-1}-b+1$ and $h_{k-1}=bh_{k-2}+1$. Furthermore, $h_k\equiv 1\pmod{b^2}$ because $k$ is even. Since $n\equiv h_k\equiv 1\pmod{b^2}$, we have by Lemma \ref{Lem2} that \[s_b(n)=s_b\left(\frac{n-1}{b^2}\right)+s_b\left(\frac{n+b-1}{b}\right)\] \[=s_b\left(jb^{k-2}+\frac{h_k-1}{b^2}\right)+s_b\left(jb^{k-1}+\frac{h_k-1}{b}+1\right)\] \[=s_b\left(jb^{k-2}+\frac{h_{k-1}-1}{b}\right)+s_b\left(jb^{k-1}+h_{k-1}\right)\] \[=s_b\left(jb^{k-2}+h_{k-2}\right)+s_b\left(jb^{k-1}+h_{k-1}\right).\] By induction on $k$, $s_b\left(jb^{k-2}+h_{k-2}\right)=F_k$ and $s_b\left(jb^{k-1}+h_{k-1}\right)=F_{k+1}$. Thus, $s_b(n)=F_k+F_{k+1}=F_{k+2}$, as desired. 
\par 
\noindent Case 2: In this case, suppose $n\not\equiv 1\pmod b$. By Lemma \ref{Lem1}, there exists an integer $n'$ with $b^{k-1}\leq n'\leq b^k$ such that $s_b(n)=s_b(n')$. If $n'=b^k$, then, using Lemma \ref{Lem1} again, $s_b(n)=1<F_{k+2}$. If $n'<b^k$, then it follows from induction on $k$ that $s_b(n)=s_b(n')\leq F_{k+1}<F_{k+2}$. 
\par 
\noindent Case 3: In this final case, suppose $n\equiv 1\pmod b$ and that $t\neq h_k$ and $t\neq h_{k+1}$. Suppose, by way of contradiction, that $s_b(n)\geq F_{k+2}$. Since $t\equiv n\equiv 1\pmod b$, we may write $t=bt'+1$ for some integer $t'<b^{k-1}$. Using \eqref{Eq1}, we have 
\begin{equation} \label{Eq25} 
s_b(n)=s_b\left(\frac{n-1}{b}\right)+s_b\left(\frac{n-1}{b}+1\right)=s_b(jb^{k-1}+t')+s_b(jb^{k-1}+t'+1).
\end{equation} 
If $t'\not\equiv 0,1\pmod b$, then we know from Lemma \ref{Lem1} and the fact that $b^{k-1}\leq jb^{k-1}+t'<jb^{k-1}+t'+1\leq b^k$ that $s_b(jb^{k-1}+t')=s_b(v_1)$ and $s_b(jb^{k-1}+t'+1)=s_b(v_2)$ for some integers $v_1,v_2\in[b^{k-2},b^{k-1}]$. By \eqref{Eq25} and induction on $k$, it follows that if $t'\not\equiv 0,1\pmod b$, then \[s_b(n)=s_b(v_1)+s_b(v_2)=\leq F_k+F_k<F_{k+2}.\] This is a contradiction, so $t'\equiv 0,1\pmod b$. We will assume that $t'\equiv 0\pmod b$; a similar argument may be used to derive a contradiction in the case $t\equiv 1\pmod b$. Write $t'=bt''$. Because $t'<b^{k-1}$ and $j\leq b-1$, $jb^{k-1}+t'+1\leq b^k$. We know that $jb^{k-1}+t'+1\neq b^k$ because $t'\equiv 0\pmod b$. Therefore, we have the inequalities $b^{k-2}\leq jb^{k-2}+t''< b^{k-1}<jb^{k-1}+t'+1<b^k$. We see by \eqref{Eq25} and induction that \[s_b(n)=s_b(jb^{k-1}+t')+s_b(jb^{k-1}+t'+1)\] \[=s_b(jb^{k-2}+t'')+s_b(jb^{k-1}+t'+1)\leq F_k+F_{k+1}=F_{k+2},\] where the equality $s_b(jb^{k-1}+t')=s_b(jb^{k-2}+t'')$ is immediate from \eqref{Eq1}. This last inequality must be an equality since we are assuming $s_b(n)\geq F_{k+2}$, so we must have $s_b(jb^{k-2}+t'')=F_k$ and $s_b(jb^{k-1}+t'+1)=F_{k+1}$. The induction hypothesis states that this is only possible if $t''\in\{h_{k-2},h_{k-1}\}$ and $t'+1\in\{h_{k-1},h_k\}$. Suppose first that $t''=h_{k-2}$ and $t'+1=h_{k-1}$. We have $h_{k-1}=bh_{k-2}+1$, so it follows from \eqref{Eq2} that $k$ must be even. Using \eqref{Eq2} again, we see that $h_k=bh_{k-1}-b+1=b(t'+1)-b+1=t$, which contradicts our assumption that $t\neq h_k$. Similarly, if $t''=h_{k-1}$ and $t'+1=h_k$, then we may derive the contradiction $t=h_{k+1}$. It is clearly impossible to have $t''=h_{k-1}$ and $t'+1=h_{k-1}$ since $t'=bt''$. Therefore, we are left to conclude that $t''=h_{k-2}$ and $t'+1=h_k$. If $k$ is even, then we may use \eqref{Eq2} to write \[bt''=t'=h_k-1=bh_{k-1}-b=b(bh_{k-2}+1)-b=b^2h_{k-2}=b^2t'',\] which is impossible. This means that $k$ must be odd, so $h_k=bh_{k-1}+1$ by \eqref{Eq2}. Since $h_k=t'+1=bt''+1=bh_{k-2}+1$, we conclude that $h_{k-1}=h_{k-2}$. It is easy to see from Definition \ref{Def2} that this is only possible if $k=3$. Hence, $t''=h_1=1$, $t'=h_3-1=b$, and $t=bt'+1=b^2+1$. However, this means that $t=h_4=h_{k+1}$, which is our final contradiction because we assumed $t\not\in\{h_k,h_{k+1}\}$.     
\end{proof} 
Now that we know the maximum values of $s_b(n)$ for $b^k\leq n<b^{k+1}$, we may easily derive the following result. 
\begin{corollary} \label{Cor1}
We have \[\limsup_{n\rightarrow\infty}\frac{s_b(n)}{n^{\log_b\phi}}\geq\frac{(b^2-1)^{\log_b\phi}}{\sqrt 5}.\]
\end{corollary}
\begin{proof} 
We will need Binet's formula for the Fibonacci numbers, which states that $F_m=\dfrac{\phi^m-(-\phi)^{-m}}{\sqrt5}$ for all $m\in\mathbb Z$. For each positive integer $k$, let $u_k=b^k+h_k$. Let $i_k=(b^2-1)h_k-b^k$, and observe that \[i_k=\begin{cases} -1, & \mbox{if } 2\vert k; \\ b^2-b-1, & \mbox{if } 2\nmid k. \end{cases}\]  We have \[\lim_{k\rightarrow\infty}\frac{\phi^{k+2}-(-\phi)^{-(k+2)}}{(b^{k+2}+i_k)^{\log_b\phi}}=\lim_{k\rightarrow\infty}\frac{1-\phi^{-(k+2)}(-\phi)^{-(k+2)}}{\phi^{-(k+2)}(b^{k+2}+i_k)^{\log_b\phi}}\]\[=\lim_{k\rightarrow\infty}\frac{1-\phi^{-(k+2)}(-\phi)^{-(k+2)}}{(1+i_k/b^{k+2})^{\log_b\phi}}=1.\] By Proposition \ref{Prop1} and Binet's formula, \[\frac{s_b(u_k)}{u_k^{\log_b\phi}}=\frac{F_{k+2}}{(b^k+h_k)^{\log_b\phi}}=\frac{\phi^{k+2}-(-\phi)^{-(k+2)}}{(b^k+h_k)^{\log_b\phi}\sqrt 5}=\frac{\phi^{k+2}-(-\phi)^{-(k+2)}}{(b^k+\frac{b^k+i_k}{b^2-1})^{\log_b\phi}\sqrt 5}\] \[=\frac{(b^2-1)^{\log_b\phi}}{\sqrt 5}\frac{\phi^{k+2}-(-\phi)^{-(k+2)}}{(b^{k+2}+i_k)^{\log_b\phi}},\] so \[\limsup_{n\rightarrow\infty}\frac{s_b(n)}{n^{\log_b\phi}}\geq\lim_{k\rightarrow\infty}\frac{s_b(u_k)}{u_k^{\log_b\phi}}=\frac{(b^2-1)^{\log_b\phi}}{\sqrt 5}.\]
\end{proof}  
We now wish to show that $\displaystyle{\limsup_{n\rightarrow\infty}\frac{s_b(n)}{n^{\log_b\phi}}\leq\frac{(b^2-1)^{\log_b\phi}}{\sqrt 5}}$. If $b^k+h_k<n<b^{k+1}$ for some positive integer $k$, then we know from Proposition \ref{Prop1} that \[\frac{s_b(n)}{n^{\log_b\phi}}\leq\frac{F_{k+2}}{n^{\log_b\phi}}=\frac{s_b(b^k+h_k)}{n^{\log_b\phi}}<\frac{s_b(b^k+h_k)}{(b^k+h_k)^{\log_b\phi}}.\] We saw in the proof of Corollary \ref{Cor1} that \[\lim_{k\rightarrow\infty}\frac{s_b(b^k+h_k)}{(b^k+h_k)^{\log_b\phi}}=\frac{(b^2-1)^{\log_b\phi}}{\sqrt 5}.\] Hence, we need only find a sufficiently strong upper bound for $s_b(n)$ when $b^k\leq n\leq b^k+h_k$ for some positive integer $k$. For this purpose, we make the following definitions. 
\begin{definition} \label{Def3}  
For nonnegative integers $k,r,y$, let \[I(k,r,y)=\left\{n\in\mathbb N\colon b^k<n\leq b^k+\sum_{i=0}^yb^{r-2i}\right\},\] \[\mu(k,r,y)=\max\{s_b(n)\colon n\in I(k,r,y)\},\] and \[\nu(k,r,y)=\min\{n\in I(k,r,y)\colon s_b(n)=\mu(k,r,y)\}.\] 
\end{definition} 
Our goal is to calculate $\mu(k,r,y)$ for any given nonnegative integers $k,r,y$ that satisfy $2y<r<k-1$. This will allow us to to derive tight upper bounds for all integers $n$ that satisfy $b^k\leq n\leq b^k+h_k$ for some $k$ by choosing appropriate values of $r$ and $y$. One might think that a simple inductive argument based on the recurrence relation \eqref{Eq1} should be able to derive our upper bounds quite effortlessly. However, the author has found that attempts to prove upper bounds for $s_b(n)$ using induction often fail or become incredibly convoluted. Indeed, the reader may wish to look at \cite{Coons14} in order to appreciate the surprising amount of ingenuity that Coons and Tyler need for the derivation of their upper bounds (which are weaker than ours) in the specific case $b=2$. Therefore, we shall prove a sequence of lemmas in order to develop more combinatorial means of calculating $\mu(k,r,y)$ and $\nu(k,r,y)$. 
\begin{lemma} \label{Lem3}
Let $a_ta_{t-1}\cdots a_0$ be the ordinary base $b$ expansion of a positive integer $n$. Let $c_\ell c_{\ell-1}\cdots c_0$ be a base $b$ over-expansion of $n$. Set $c_i=0$ for all integers $i$ with $\ell<i\leq t$. For any $m\in\{0,1,\ldots,t\}$, \[c_m\in\begin{cases} \{0,b-1,b\}, & \mbox{if } a_m=0; \\ \{0,1,b\}, & \mbox{if } a_m=1; \\ \{a_m-1,a_m\}, & \mbox{otherwise} . \end{cases}\]
\end{lemma} 
\begin{proof} 
We know that $\displaystyle{\sum_{i=0}^ta_ib^i=\sum_{i=0}^t c_ib^i}$ because $a_ta_{t-1}\cdots a_0$ and $c_\ell c_{\ell-1}\cdots c_0$ are base $b$ over-expansions of the same number. Therefore, for any $j\in\{0,1,\ldots,t\}$, \begin{equation}\label{Eq4} 
\sum_{i=0}^ja_ib^i\equiv\sum_{i=0}^jc_ib^i\pmod{b^{j+1}}. 
\end{equation} 
If $\displaystyle{\sum_{i=0}^ja_ib^i>\sum_{i=0}^jc_ib^i}$ for some $j\in\{0,1,\ldots,t\}$, then it follows from \eqref{Eq4} and the fact that $a_i\in\{0,1,\ldots,b-1\}$ for all $i$ that \[\sum_{i=0}^jc_ib^i\leq-b^{j+1}+\sum_{i=0}^ja_ib^i\leq-b^{j+1}+\sum_{i=0}^j(b-1)b^i=-1,\] which is impossible. Hence, $\displaystyle{\sum_{i=0}^ja_ib^i\leq\sum_{i=0}^jc_ib^i}$ for all $j\in\{0,1,\ldots,t\}$. Choose some $m\in\{0,1,\ldots,t\}$. Since $c_i\leq b$ for all $i$, we have \begin{equation}\label{Eq5} \sum_{i=0}^mc_ib^i\leq c_mb^m+\sum_{i=0}^{m-1}b\cdot b^i=(c_m+1)b^m+\sum_{i=0}^{m-2}b^{i+1}<(c_m+2)b^m. 
\end{equation} It follows that \[a_mb^m\leq\sum_{i=0}^ma_ib^i\leq\sum_{i=0}^mc_ib^i<(c_m+2)b^m,\] so $c_m\geq a_m-1$. To complete the proof, we simply need to show that if $c_m>a_m$, then either $a_m=0$ and $b_m\in\{b-1,b\}$ or $a_m=1$ and $c_m=b$.
Suppose $c_m>a_m$. We have \[\sum_{i=0}^ma_ib^i=a_mb^m+\sum_{i=0}^{m-1}a_ib^i\leq a_mb^m+\sum_{i=0}^{m-1}c_ib^i<\sum_{i=0}^mc_ib^i\] and \[\sum_{i=0}^ma_ib^i\equiv\sum_{i=0}^mc_ib^i\pmod{b^{m+1}}\] (by \eqref{Eq4}), so we know from \eqref{Eq5} that \[\sum_{i=0}^ma_ib^i\leq -b^{m+1}+\sum_{i=0}^mc_ib^i<-b^{m+1}+(c_m+2)b^m.\] This implies that $a_mb^m<-b^{m+1}+c_mb^m+2b^m$, so $c_m-a_m>b-2$. This is impossible if $a_m\not\in\{0,1\}$ because $c_m\leq b$. If $a_m=0$, we must have either $c_m=b-1$ or $c_m=b$, and if $a_m=1$, then we must have $c_m=b$. 
\end{proof} 
\begin{lemma} \label{Lem4} 
Let $k,r,y$ be nonnegative integers with $2y<r<k-1$. If $a_ta_{t-1}\cdots a_0$ is the ordinary base $b$ expansion of $\nu(k,r,y)-1$, then $t=k$, $a_k=1$, $a_j=a_0=0$ for all $j\in\{r+1,\ldots,k-1\}$, and $a_i\in\{0,1\}$ for all $i\in\{0,1,\ldots,k\}$.  
\end{lemma} 
\begin{proof} 
Let $\nu=\nu(k,r,y)$, and let $a_ta_{t-1}\cdots a_0$ be the ordinary base $b$ expansion of $\nu-1$. It follows from the fact that $\nu\in I(k,r,y)$ that $t=k$, $a_k=1$, and $a_j=0$ for all $j\in\{r+1,\ldots,k-1\}$. We still need to show that $a_0=0$ and $a_i\in\{0,1\}$ for all $i\in\{0,1,\ldots,k\}$. Suppose, for the sake of finding a contradiction, that $a_m\in\{2,3,\ldots,b-1\}$ for some $m\in\{0,1,\ldots,k\}$. Because $a_k=1$, we know that $m\in\{0,1,\ldots,k-1\}$. Let $\nu'=\nu-(a_m-1)b^m$. The ordinary base $b$ expansion of $\nu'-1$ is simply the word obtained from $a_ka_{k-1}\cdots a_0$ by replacing the $m^{th}$ digit (which is $a_m$) with $1$. Observe that $\nu'<\nu$ and $\nu'\in I(k,r,y)$ since \[b^k<b^k+b^m=b^k+a_mb^m-(a_m-1)b^m\leq\nu-(a_m-1)b^m=\nu'.\] Hence, by the definition of $\nu$ in Definition \ref{Def3}, $s_b(\nu')<s_b(\nu)$. Choose some base $b$ over-expansion $c_\ell c_{\ell-1}\cdots c_0$ of $\nu-1$, and let $f(c_\ell c_{\ell-1}\cdots c_0)$ be the word obtained from $c_\ell c_{\ell-1}\cdots c_0$ by replacing the $m^{th}$ digit (which is $c_m$) with $c_m-a_m+1$. Lemma \ref{Lem3} implies that $c_m-a_m+1\in\{0,1\}$ because $c_m\in\{a_m-1,a_m\}$. Therefore, $f(c_\ell c_{\ell-1}\cdots c_0)$ is a base $b$ over-expansion of $\nu'-1$. We see that $f$ is an injection from the set of base $b$ over-expansions of $\nu-1$ to the set of base $b$ over-expansions of $\nu'-1$. This contradicts the fact that $s_b(\nu')<s_b(\nu)$, so we conclude that $a_i\in\{0,1\}$ for all $i\in\{0,1,\ldots,k\}$. 
\par 
We are left with the task of showing that $a_0=0$. Suppose $a_0\neq 0$. By the preceding paragraph, we must have $a_0=1$. This means that $\nu\equiv 2\pmod b$, so \[s_b(\nu)=s_b\left(\frac{\nu-2}{b}+1\right)<s_b\left(\frac{\nu-2}{b}\right)+s_b\left(\frac{\nu-2}{b}+1\right)=s_b(\nu-1)\] by \eqref{Eq1}. Since $\nu\equiv 2\pmod b$ and $\nu>b^k$ by definition, we know that $\nu-1>b^k$. Therefore, $\nu-1\in I(k,r,y)$ and $s_b(\nu)<s_b(\nu-1)$, which contradicts the definition of $\nu$. It follows from this contradiction that $a_0=0$. 
\end{proof}
Lemma \ref{Lem4} hints that it is of interest to enumerate the base $b$ over-expansions of positive integers whose ordinary base $b$ expansions use only the digits $0$ and $1$. Let $e_0,e_1,\ldots,e_{\ell}$ be nonnegative integers with $e_0<e_1<\cdots<e_\ell$, and let $n=b^{e_0}+b^{e_1}+\cdots+b^{e_\ell}$. Let $Q_i$ denote the operation that changes one base $b$ over-expansion of $n$ into another by increasing the $i^{th}$ digit of an expansion by $b$ and decreasing the $(i+1)^{th}$ digit of the expansion by $1$. The operation $Q_i$ can only be used if the $i^{th}$ digit of the expansion under consideration is a $0$. If we use the operation $Q_i$ when the $(i+1)^{th}$ digit of the expansion is a $0$, then this digit is immediately converted to a $b-1$ and the $(i+2)^{th}$ digit is reduced by $1$. If the $(i+2)^{th}$ digit is also $0$, then it is immediately converted to a $b-1$ and the $(i+3)^{th}$ digit is reduced by $1$. This process continues until a nonzero digit is reduced by $1$. After using the operation $Q_i$, any leading $0$'s are erased from the expansion. The transformation of $0$'s to $(b-1)$'s is analogous to the transformation of $0$'s to $9$'s that occurs when the number $1$ is subtracted from the number $10000$, resulting in $9999$. As an example, only the operations $Q_2,Q_4,Q_5,Q_6$, and $Q_8$ may be used to transform the expansion $1010001011$ into a new expansion. The operation $Q_4$ changes the expansion $1010001011$ into the expansion $100(b-1)(b-1)b1011$. The operation $Q_8$ first changes the expansion $1010001011$ into $0b10001011$; the leading $0$ is then erased, yielding $b10001011$.  
\par  
We will transform the ordinary base $b$ expansion of $n$ into a base $b$ over-expansion $c_tc_{t-1}\cdots c_0$ of $n$ by repeated use of the operations just described. At each step, we will choose the value of $c_{e_i}$ for some $i\in\{0,1,\ldots,\ell\}$ while simultaneously deciding the values of $c_j$ for all $j\in\{e_{i-1}+1,e_{i-1}+2,\ldots,e_i-1\}$ (or all $j<e_i$ in the case $i=0$). We proceed by permanently deciding the values of $c_0,c_1,\ldots,c_{e_0}$, then permanently deciding the values of $c_{e_0+1},c_{e_0+2},\ldots,c_{e_1}$, and so on until we decide the values of $c_{e_{\ell-1}+1},c_{e_{\ell-1}+2},\ldots,c_{e_\ell}$. We omit $c_{e_\ell}$ from the word $c_tc_{t-1}\cdots c_0$ if $c_{e_\ell}=0$; in this case, $t=e_\ell-1$. For the sake of providing a concrete example of this process, we will suppose that $\{e_0,e_1,\ldots,e_\ell\}=\{2,4,5,9,11\}$ and $b=7$. Here, the ordinary base $7$ expansion of $n$ is $101000110100$.
Since $e_0=2$, we first choose the value of $c_2$ (while simultaneously deciding the values of $c_0$ and $c_1$). There is only one way to set $c_2=1$; namely, we keep the same expansion $101000110100$. If we want to have $c_2=0$, then we could either perform the operation $Q_0$ to get the expansion $101000110067$ or perform the operation $Q_1$ to get the expansion $101000110070$. Similarly, if we want to have $c_2=7$, then we could either perform the operations $Q_0$ and then $Q_2$ to obtain $101000106767$ or perform the operations $Q_1$ and then $Q_2$ to get $101000106770$. That is, there are $e_0=2$ ways to set $c_2=0$ and $e_0=2$ ways to set $c_2=7$. By Lemma \ref{Lem3}, $0,1,$ and $7$ are the only possible values of $c_2$. For this example, we will assume we chose to use the operations $Q_1$ and $Q_2$ to get the expansion $101000106770$. Next, we choose the value of $c_4$ because $e_1=4$. Because we used the operation $Q_2$, the $4^{th}$ digit of the expansion was temporarily converted into a $0$. This means that there is no way to set $c_4=1$ using the operations $Q_i$. There is one way to set $c_4=0$; just keep the expansion $101000106770$. Similarly, there is one way to set $c_4=7$; just use the operation $Q_4$ to obtain $101000076770$. We will assume that we make the former choice and keep the expansion $101000106770$. Next, we choose the value of $c_5$ since $e_2=5$. The only way to set $c_5=1$ is to keep the expansion the same. Since we have already determined $c_0,c_1,c_2,c_3,c_4$, we cannot perform any of the operations $Q_0,Q_1,Q_2,Q_3,Q_4$ in order to temporarily reduce the $5^{th}$ digit of the expansion by $1$. In other words, $c_5$ is stuck with the value $1$. We now have the expansion $101000106770$, and we wish to choose the value of $c_9$ (since $e_3=9$) while simultaneously determining the values of $c_6,c_7,c_8$. Again, the only way to make $c_9=1$ is to keep the expansion $101000106770$. To make $c_9=0$, we could use $Q_6$ to get $100667106770$, use $Q_7$ to get $100670106770$, or use $Q_8$ to get $100700106770$. To set $c_9=7$, we could use $Q_6$ and then $Q_9$ to obtain $67667106770$, use $Q_7$ and then $Q_9$ to get $67670106770$, or use $Q_8$ and then $Q_9$ to get $67700106770$. Hence, there are $e_3-e_2-1=3$ ways to set $c_9=0$ and $e_3-e_2-1=3$ ways to set $c_9=7$. We will assume that we choose to use $Q_7$ (and not $Q_9$) to obtain the expansion $100670106770$. All that is left to do is decide the value of $c_{11}$ (while simultaneously determining $c_{10}$). We cannot set $c_{11}=7$ because there is no nonzero digit to the left of the $11^{th}$ digit from which to ``borrow." Alternatively, one could observe that if $c_{11}=7$, then $c_{11}c_{10}\cdots c_0$ would be a base $7$ over-expansion for a number strictly larger than $n$. We can, however, keep the expansion      
$100670106770$ if we wish to set $c_{11}=1$. There is one way (because $e_4-e_3-1=1$) to set $c_{11}=0$; simply use the operation $Q_{10}$ to obtain the expansion $70670106770$. 
\begin{figure}  
\includegraphics[height=68mm]{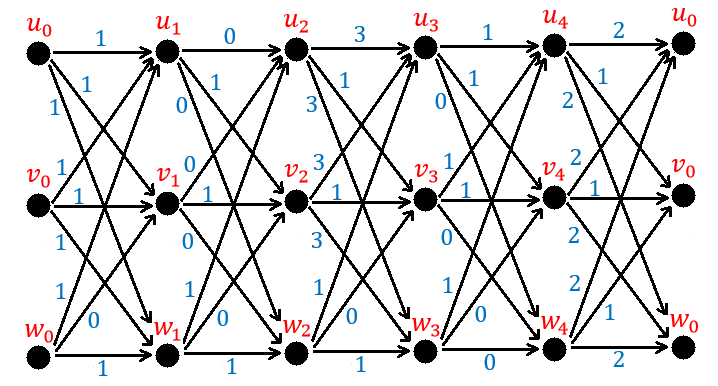}
\caption{The edge-weighted digraph $G$ for the given example. Note that each of the vertices $u_0,v_0,w_0$ is drawn twice.} \label{Fig1}
\end{figure} 
\par
Figure \ref{Fig1} depicts an edge-weighted digraph $G$ which encodes all the possible choices that we could have made in this example. The graph has only fifteen vertices, but we have drawn each of the vertices $u_0,v_0,w_0$ twice in order to improve the aesthetics of the image. The vertex $u_i$ corresponds to choosing to set $c_{e_i}=0$. The vertex $v_i$ corresponds to setting $c_{e_i}=1$. The vertex $w_i$ corresponds to setting $c_{e_i}=7$. The weights of the edges correspond to the number of choices possible. For example, if we have chosen to let $c_5=0$ (corresponding to the vertex $u_2$), then there are three ways to set $c_9=7$ (corresponding to the vertex $w_3$). Thus, there is an edge of weight $3$ from $u_2$ to $w_3$. After setting $c_9=7$, it is impossible to set $c_{11}=1$, so there is an edge of weight $0$ from $w_3$ to $v_4$. The reader might ask why there are edges from $u_4,v_4,w_4$ to $u_0,v_0,w_0$. We include these edges because we wish to interpret base $7$ over-expansion of $n$ in terms of closed walks in the graph $G$. There are edges of weight $2$ from $u_4,v_4,w_4$ to $u_0$ because there are two ways to set $c_2=0$, regardless of the value of $c_{11}$ (recall that we choose $c_2$ before choosing $c_{11}$). Similarly, there are edges of weight $1$ from $u_4,v_4,w_4$ to $v_0$ and edges of weight $2$ from $u_4,v_4,w_4$ to $w_0$. 
\par 
Suppose we want to construct a base $7$ over-expansion of $n$ in which $c_2=0$, $c_4=1$, $c_5=1$, $c_9=7$, and $c_{11}=0$. This choice of the values of $c_{e_i}$ for $i\in\{0,1,2,3,4\}$ corresponds to the closed walk $(u_0,v_1,v_2,w_3,u_4,u_0)$ in $G$. The weight of this walk (which we calculate as the product of the weights of its edges) is $6$, so there are $6$ base $7$ over-expansions of $n$ with these specific values of $c_2,c_4,c_5,c_9,c_{11}$. As another example, there are no base $7$ over-expansions of $n$ in which $c_2=7$, $c_4=7$, $c_5=1$, $c_9=0$, and $c_{11}=1$ because the weight of the closed walk $(w_0,w_1,v_2,u_3,v_4,w_0)$ is $0$. 
\par 
We are finally in a position to enumerate the base $7$ over-expansions of $n$. To do so, we will need the following definition. 
\begin{definition} 
For any real $t$, let \[M_t=\left( \begin{array}{ccc}
t & 1 & t \\ 
t & 1 & t \\
0 & 1 & 0 \end{array} \right)\] and \[N_t=\left( \begin{array}{ccc}
t & 1 & t \\
t & 1 & t \\
0 & 0 & 0 \end{array} \right).\]
\end{definition} 
If we put the vertices of $G$ in the order $u_0,v_0,w_0,u_1,v_1,w_1,\ldots,$ $u_4,v_4,w_4$, then the adjacency matrix of $G$ (written as a block matrix) is
\[A=\left( \begin{array}{ccccc}
O & M_{d_1} & O & O & O \\
O & O & M_{d_2} & O & O \\
O & O & O & M_{d_3} & O \\
O & O & O & O & B \\
C & O & O & O & O
\end{array} \right),\] where $d_i=e_i-e_{i-1}-1$, \[B=\left( \begin{array}{ccc}
e_\ell-e_{\ell-1}-1 & 1 & 0 \\
e_\ell-e_{\ell-1}-1 & 1 & 0 \\
1 & 0 & 0 \end{array} \right)=\left( \begin{array}{ccc}
1 & 1 & 0 \\
1 & 1 & 0 \\
1 & 0 & 0 \end{array} \right),\] \[C=\left( \begin{array}{ccc}
e_0 & 1 & e_0 \\
e_0 & 1 & e_0 \\
e_0 & 1 & e_0 \end{array} \right)=\left( \begin{array}{ccc}
2 & 1 & 2 \\
2 & 1 & 2 \\
2 & 1 & 2 \end{array} \right),\] and $O$ denotes the $3\times 3$ zero matrix. The total number of base $7$ over-expansions of $n$ is equal to the sum of the weights of the closed walks of length $5$ in $G$ that start at $u_0$, $v_0$, or $w_0$. This, in turn, is equal to the sum of the first three diagonal entries of $A^5$. Using elementary linear algebra, we see that the first three rows and the first three columns of $A^5$ intersect in a $3\times 3$ block given by $M_1M_2M_3BC$. Therefore, the sum of the first three diagonal entries of $A^5$ is $\Tr(M_1M_2M_3BC)$, where $\Tr$ denotes the trace of a matrix. One easily calculates this value to be $158$. 
\par 
We now state this result in greater generality (and in a slightly different form). Sketching the proof of the following theorem, we trust the reader to see that the method used in the preceding example is representative of the method used in general. 
\begin{theorem} \label{Thm2} 
Let $e_0,e_1,\ldots,e_\ell$ be nonnegative integers with $e_0<e_1<\cdots<e_\ell$ (where $\ell\geq 1$). Let $d_i=e_i-e_{i-1}-1$ for all $i\in\{1,2,\ldots,\ell\}$. The number of base $b$ over-expansions of the number $b^{e_0}+b^{e_1}+\cdots+b^{e_\ell}$ is given by \[\Tr(N_{e_0}M_{d_1}M_{d_2}\cdots M_{d_\ell}).\] 
\end{theorem} 
\begin{proof} 
Let $n=b^{e_0}+b^{e_1}+\cdots+b^{e_\ell}$. Let \[A=\left( \begin{array}{ccccccc}
O & M_{d_1} & O & \cdots & O & O & O \\
O & O & M_{d_2} & \cdots & O & O & O \\
O & O & O & \cdots & O & O & O \\
\vdots & \vdots & \vdots & \ddots & \vdots & \vdots & \vdots \\
O & O & O & \cdots & O & M_{d_{\ell-1}} & O \\
O & O & O & \cdots & O & O & B \\
C & O & O & \cdots & O & O & O
\end{array} \right),\] where \[B=\left( \begin{array}{ccc}
d_\ell & 1 & 0 \\
d_\ell & 1 & 0 \\
1 & 0 & 0 \end{array} \right)\hspace{.4cm}\text{and}\hspace{.5cm}C=\left( \begin{array}{ccc}
e_0 & 1 & e_0 \\
e_0 & 1 & e_0 \\
e_0 & 1 & e_0 \end{array} \right).\] Let $G$ be the edge-weighted digraph with vertex set $\{u_0,v_0,w_0,u_1,v_1,w_1,\ldots,$ $u_\ell,v_\ell,w_\ell\}$ and adjacency matrix $A$. Suppose we wish to construct a base $b$ over-expansion $c_tc_{t-1}\cdots c_0$ of $n$. As in the example above, specifying the values of $c_{e_0},c_{e_1},\ldots,c_{e_\ell}$ (each such value must be $0$, $1$, or $b$ by Lemma \ref{Lem3}) corresponds to choosing a closed walk of length $\ell+1$ in $G$ starting at $u_0$, $v_0$, or $w_0$. The weight of this walk is the number of base $b$ over-expansions $c_{e_\ell}c_{e_\ell-1}\cdots c_0$ of $n$ that have the specified values of $c_{e_0},c_{e_1},\ldots,c_{e_\ell}$. Therefore, the total number 
of base $b$ over-expansions of $n$ is equal to the sum of the first three diagonal entries of $A^{\ell+1}$. The first three rows and the first three columns of $A^{\ell+1}$ intersect in the $3\times 3$ block $M_{d_1}M_{d_2}\cdots M_{d_{\ell-1}}BC$, so the number of base $b$ 
over-expansions of $n$ is $\Tr(M_{d_1}M_{d_2}\cdots M_{d_{\ell-1}}BC)$. Now, one may easily calculate that $BC=M_{d_\ell}N_{e_0}$. Therefore, using the fact that $\Tr(XY)=\Tr(YX)$ for any square matrices $X,Y$ of the same 
size, we conclude that \[\Tr(M_{d_1}M_{d_2}\cdots M_{d_{\ell-1}}BC)=\Tr(M_{d_1}M_{d_2}\cdots M_{d_{\ell-1}}M_{d_{\ell}}N_{e_0})\] \[=\Tr(N_{e_0}M_{d_1}M_{d_2}\cdots M_{d_\ell}).\] 
\end{proof}             
\begin{definition} 
We define $\Xi=\left\{\left( \begin{array}{ccc}
x_1 & z & x_1 \\
x_1 & z & x_1 \\
x_2 & x_3 & x_2 \end{array} \right)\colon x_1,x_2,x_3,z\geq 0, z\neq 0\right\}$ and $\Xi'=\Xi\cup\{I_3\}$, where $I_3$ is the $3\times 3$ identity matrix.
\end{definition}   
We omit the proofs of the following three lemmas because they are fairly straightforward.  
\begin{lemma} \label{Lem6} 
The set $\Xi$ is closed under matrix multiplication, and $\Xi'$ is a monoid under matrix multiplication. For any nonnegative real number $u$, $M_u,N_u\in\Xi$. 
\end{lemma}
\begin{lemma} \label{Lem7}  
Let $X$ be a $3\times 3$ matrix with nonnegative real entries, and let $Y\in\Xi'$. If the second diagonal entry (the entry in the second row and the second column) of $X$ is positive, then $\Tr(XY)>0$. 
\end{lemma} 
\begin{lemma} \label{Lem8}
For any positive integer $t$, \[M_1^t=\left( \begin{array}{ccc}
F_{2t} & F_{2t-1} & F_{2t} \\
F_{2t} & F_{2t-1} & F_{2t} \\
F_{2t-1} & F_{2t-2} & F_{2t-1} \end{array} \right).\]
\end{lemma} 
The following lemma attempts to gain information about the ordinary base $b$ expansion of $\nu(k,r,y)-1$ when $2y<r<k-1$. By Lemma \ref{Lem4}, all of the digits in that expansion are $0$'s and $1$'s, so we may write $\nu(k,r,y)-1=b^{e_0}+b^{e_1}+\cdots+b^{e_\ell}$ for some nonnegative integers $e_0,e_1,\ldots,e_\ell$ that satisfy $e_0<e_1<\cdots<e_\ell=k$. Because $\nu(k,r,y)\in I(k,r,y)$, we have $b^k\leq\nu(k,r,y)-1<b^k+\displaystyle{\sum_{i=0}^yb^{r-2i}}$. It follows that $e_{\ell-1-j}\leq r-2j$ for all $j\in\{0,1,\ldots,y\}$, where equality cannot hold for all $j$. Hence, we may consider the smallest nonnegative integer $\lambda$ such that $e_{\ell-1-\lambda}\leq r-2\lambda-1$. 
\begin{lemma} \label{Lem9}
Let $k,r,y$ be nonnegative integers with $2y<r<k-1$. Let $a_ka_{k-1}\cdots a_0$ be the ordinary base $b$ expansion of $\nu(k,r,y)-1$, and write $\nu(k,r,y)-1=b^{e_0}+b^{e_1}+\cdots+b^{e_\ell}$, where $e_0,e_1,\ldots,e_\ell$ are nonnegative integers that satisfy $e_0<e_1<\cdots<e_\ell=k$. Let $\lambda$ be the smallest nonnegative integer such that $e_{\ell-1-\lambda}\leq r-2\lambda-1$. There does not exist $i\in\{1,2,\ldots,e_{\ell-1-\lambda}\}$ such that $a_i=a_{i+1}$.
\end{lemma} 
\begin{proof} 
Let $\nu=\nu(k,r,y)$. Note that it follows from the preceding paragraph that $\lambda\leq y$. Because $\nu-1=b^{e_0}+b^{e_1}+\cdots+b^{e_\ell}$, we have 
\begin{equation}\label{Eq6} 
a_i=\begin{cases} 1, & \mbox{if } i=e_j\text{ for some }j; \\ 0, & \mbox{otherwise} . \end{cases} 
\end{equation} for all $i\in\{0,1,\ldots,k\}$. In particular, $e_0>0$ because $a_0=0$ by Lemma \ref{Lem4} We will let $d_j=e_j-e_{j-1}-1$ for all $j\in\{1,2,\ldots,\ell\}$. Suppose, by way of contradiction, that $a_i=a_{i+1}$ for some $i\in\{1,2,\ldots,e_{\ell-1-\lambda}\}$, and let $m$ be the smallest such index $i$. We have five cases to consider. 
\par 
First, assume $a_1=1$ and $a_m=a_{m+1}=0$. By the minimality of $m$, we see that $a_i=0$ for all nonnegative even integers $i<m$ and $a_j=1$ for all positive odd integers $j<m$. In addition, $m$ must be even. Therefore, setting $t=\dfrac{m-2}{2}$, we have $e_j=2j+1$ for all $j\in\{0,1,\ldots,t\}$. This means that $d_j=1$ for all $j\in\{1,2,\ldots,t\}$. It follows from \eqref{Eq6} and the assumption that $a_m=a_{m+1}=0$ that 
\begin{equation} \label{Eq26}
e_t=m-1<m<m+1<e_{t+1}<e_{t+2}<\cdots<e_\ell=k.
\end{equation} 
Note that $t<\ell-1-\lambda$ because $e_t<m\in\{1,2,\ldots,e_{\ell-1-\lambda}\}$. Let us write $\nu_1=1+\displaystyle{\sum_{i=0}^tb^{e_i+1}+\sum_{i=t+1}^\ell b^{e_i}}$. We have \[\nu_1=1+\sum_{i=0}^tb^{e_i+1}+\sum_{i=t+1}^\ell b^{e_i}\leq1+\sum_{n=1}^mb^n+\sum_{i=t+1}^\ell b^{e_i}\] \[\leq1+\sum_{n=1}^{e_{t+1}-2}b^n+\sum_{i=t+1}^\ell b^{e_i}<b^{e_{t+1}-1}+\sum_{i=t+1}^\ell b^{e_i}\] \[=b^{e_{t+1}-1}+\sum_{i=t+1}^{\ell-1-\lambda}b^{e_i}+\sum_{i=\ell-\lambda}^\ell b^{e_i}.\] Recall from the paragraph immediately preceding this theorem that $e_{\ell-1-j}\leq r-2j$ for all $j\in\{0,1,\ldots,y\}$. Since $\lambda$ is the smallest nonnegative integer such that $e_{\ell-1-\lambda}\leq r-2\lambda-1$, $e_{\ell-1-i}=r-2i$ for all nonnegative integers $i<\lambda$. Furthermore, $e_\ell=k$, so we find that $\displaystyle{\sum_{i=\ell-\lambda}^\ell b^{e_i}=b^k+\sum_{i=0}^{\lambda-1}b^{r-2i}}$. Therefore, \[\nu_1<b^{e_{t+1}-1}+\sum_{i=t+1}^{\ell-1-\lambda}b^{e_i}+b^k+\sum_{i=0}^{\lambda-1}b^{r-2i}\] \[<b^{e_{\ell-1-\lambda}+1}+b^k+\sum_{i=0}^{\lambda-1}b^{r-2i}\leq b^k+b^{r-2\lambda}+\sum_{i=0}^{\lambda-1}b^{r-2i}\leq b^k+\sum_{i=0}^yb^{r-2i},\] where we have used the inequalities $e_{\ell-1-\lambda}\leq r-2\lambda-1$ and $\lambda\leq y$. This shows that $\nu_1\in I(k,r,y)$, so $s_b(\nu_1)\leq s_b(\nu)$ by the definition of $\nu$. Our goal is to derive a contradiction by showing that $s_b(\nu_1)>s_b(\nu)$. Let $B=M_{d_{t+2}}M_{d_{t+3}}\cdots M_{d_\ell}$. We saw that $t<\ell-1-\lambda$, so $t+1<\ell$. This shows that the matrix product defining $B$ is nonempty. Lemma \ref{Lem6} implies that $B\in\Xi$. 
Now, $e_0=2(0)+1=1$. By Theorem \ref{Thm2}, \[s_b(\nu)=\Tr(N_1M_{d_1}M_{d_2}\cdots M_{d_\ell})=\Tr(N_1M_1^tM_{d_{t+1}}B).\] We may write $\nu_1-1=b^{e_0'}+b^{e_1'}+\ldots+b^{e_\ell'}$, where $e_j'=e_j+1$ for all $j\in\{0,1,\ldots,t\}$ and $e_j'=e_j$ for all $j\in\{t+1,\ldots,\ell\}$. Setting $d_j'=e_j'-e_{j-1}'$ for all $j\in\{1,2,\ldots,\ell\}$, we have \[s_b(\nu_1)=\Tr(N_{e_0'}M_{d_1'}M_{d_2'}\cdots M_{d_\ell'})=\Tr(N_2M_1^tM_{d_{t+1}-1}B)\] by Theorem \ref{Thm2}. Consequently, \[s_b(\nu_1)-s_b(\nu)=\Tr((N_2M_1^tM_{d_{t+1}-1}-N_1M_1^tM_{d_{t+1}})B).\] We remark that one must take care when considering the case $t=0$. In this case, $M_1^t$ is the $3\times 3$ identity matrix. A straightforward calculation invoking Lemma \ref{Lem8} shows that \[N_2M_1^tM_{d_{t+1}-1}-N_1M_1^tM_{d_{t+1}}=F_{2t+2}N_{d_{t+1}-2},\] so \[s_b(\nu_1)-s_b(\nu)=F_{2t+2}\Tr(N_{d_{t+1}-2}B).\] Because $d_{t+1}=e_{t+1}-e_t-1\geq (m+2)-(m-1)-1=2$ by \eqref{Eq26}, we see that $N_{d_{t+1}-2}$ is a matrix with nonnegative entries whose second diagonal entry is positive. By Lemma \ref{Lem7}, $s_b(\nu_1)-s_b(\nu)>0$, which is our desired contradiction. 
\par 
Next, we assume $a_1=1$ and $a_m=a_{m+1}=1$. The proof is similar to that given in the preceding paragraph. It follows from the minimality of $m$ that $a_i=0$ for all even nonnegative integers $i<m$ and $a_j=1$ for all odd positive integers $j<m$. Also, $m$ must be odd. Let $t=\dfrac{m-1}{2}$. We have $e_j=2j+1$ for all $j\in\{0,1,\ldots,t\}$ and \[e_t=m<m+1=e_{t+1}<e_{t+2}<\cdots<e_\ell=k.\] We know from Lemma \ref{Lem4} that $a_j=0$ for all $j\in\{r+1,\ldots,k-1\}$, so $m<r$. This implies that $e_{t+1}=m+1\leq r<k=e_\ell$. Setting $\nu_2=1+\displaystyle{\sum_{i=0}^{t-1}b^{e_i+1}+\sum_{i=t+1}^\ell b^{e_i}}$, we have \[\nu_2-1=\sum_{i=0}^{t-1}b^{e_i+1}+\sum_{i=t+1}^\ell b^{e_i}=\sum_{i=0}^{t-1}b^{e_{i+1}-1}+\sum_{i=t+1}^\ell b^{e_i}=\sum_{i=1}^tb^{e_i-1}+\sum_{i=t+1}^\ell b^{e_i}\] \[<\nu-1<b^k+\sum_{i=0}^yb^{r-2i},\] so $\nu_2\in I(k,r,y)$ and $\nu_2<\nu$. It follows from the definition of $\nu$ that $s_b(\nu_2)<s_b(\nu)$. As in the preceding case, we let $B=M_{d_{t+2}}M_{d_{t+3}}\cdots M_{d_\ell}$. The matrix product defining $B$ is nonempty because $e_{t+1}<e_\ell$. In addition, $B\in\Xi$, so we may write \[B=\left(\begin{array}{ccc}
\xi_1 & \xi_2 & \xi_1 \\
\xi_1 & \xi_2 & \xi_1 \\
\xi_3 & \xi_4 & \xi_3 \end{array} \right)\] for some nonnegative real $\xi_1,\xi_2,\xi_3,\xi_4$  with $\xi_2>0$. Since $d_{t+1}=e_{t+1}-e_t-1=(m+1)-m-1=0$, we have by Theorem \ref{Thm2} that \[s_b(\nu)=\Tr(N_1M_1^tM_0B)\] and \[s_b(\nu_2)=\Tr(N_2M_1^tB).\] As in the previous paragraph, note that $M_1^t$ is the $3\times 3$ identity matrix if $t=0$. By elementary calculations that make use of Lemma \ref{Lem8}, we find that \[N_2M_1^t-N_1M_1^tM_0=F_{2t+1}\left(\begin{array}{ccc}
1 & -1 & 1 \\
1 & -1 & 1 \\
0 & 0 & 0 \end{array} \right).\] Thus, \[s_b(\nu_2)-s_b(\nu)=\Tr((N_2M_1^t-N_1M_1^tM_0)B)=F_{2t+1}(\xi_3+\xi_4)\geq 0,\] which is a contradiction. 
\par 
We now assume $a_1=0$, $a_2=1$, and $a_m=a_{m+1}=0$. Invoking the minimality of $m$, we see that $a_i=1$ for all positive even integers $i<m$ and $a_j=1$ for all positive odd integers $j<m$. The index $m$ must be odd. Let $t=\dfrac{m-3}{2}$. We find that $e_j=2j+2$ for all $j\in\{0,1,\ldots,t\}$, so $d_j=1$ for all $j\in\{1,2,\ldots,t\}$. Coupling \eqref{Eq6} with the assumption that $a_m=a_{m+1}=0$ shows us that 
\begin{equation} \label{Eq27} 
e_t=m-1<m<m+1<e_{t+1}<e_{t+2}<\cdots<e_\ell=k.
\end{equation} 
Thus, $t<\ell-1-\lambda$ because $e_t<m\in\{1,2,\ldots,e_{\ell-1-\lambda}\}$. Let $\nu_3=1+b+\displaystyle{\sum_{i=0}^tb^{e_i+1}+\sum_{i=t+1}^\ell b^{e_i}}$. We find that \[\nu_3-1=b+\displaystyle{\sum_{i=0}^tb^{e_i+1}+\sum_{i=t+1}^\ell b^{e_i}}\leq\sum_{n=1}^{e_t+1}b^n+\sum_{i=t+1}^\ell b^{e_i}\] \[\leq\sum_{n=1}^{e_{t+1}-2}b^n+\sum_{i=t+1}^\ell b^{e_i}<b^{e_{t+1}-1}+\sum_{i=t+1}^\ell b^{e_i}=b^{e_{t+1}-1}+\sum_{i=t+1}^{\ell-1-\lambda}b^{e_i}+\sum_{i=\ell-\lambda}^\ell b^{e_i}.\] Because $e_\ell=k$ and $e_{\ell-1-i}=r-2i$ for all nonnegative integers $i<\lambda$, we have $\displaystyle{\sum_{i=\ell-\lambda}^\ell b^{e_i}=b^k+\sum_{i=0}^{\lambda-1}b^{r-2i}}$. Consequently, \[\nu_3-1\leq b^{e_{t+1}-1}+\sum_{i=t+1}^{\ell-1-\lambda}b^{e_i}+b^k+\sum_{i=0}^{\lambda-1}b^{r-2i}<b^{e_{\ell-1-\lambda}+1}+b^k+\sum_{i=0}^{\lambda-1}b^{r-2i}\] \[\leq b^k+b^{r-2\lambda}+\sum_{i=0}^{\lambda-1}b^{r-2i}\leq b^k+\sum_{i=0}^yb^{r-2i},\] where we have used the inequalities $e_{\ell-1-\lambda}\leq r-2\lambda-1$ and $\lambda\leq y$. It follows that $\nu_3\in I(k,r,y)$, so $s_b(\nu_3)\leq s_b(\nu)$. Let $B=M_{d_{t+2}}M_{d_{t+3}}\cdots M_{d_\ell}$. As in the previous two cases, the matrix product defining $B$ is nonempty (because $t+1<\ell-\lambda\leq\ell$) and $B\in\Xi$. Note that $e_0=2(0)+2=2$.
We have \[s_b(\nu)=\Tr(N_2M_1^tM_{d_{t+1}}B)\] and \[s_b(\nu_3)=\Tr(N_1M_1^{t+1}M_{d_{t+1}-1}B),\] so \[s_b(\nu_3)-s_b(\nu)=\Tr((N_1M_1^{t+1}M_{d_{t+1}-1}-N_2M_1^tM_{d_{t+1}})B).\] Using Lemma \ref{Lem8}, one may easily show that \[N_1M_1^{t+1}M_{d_{t+1}-1}-N_2M_1^tM_{d_{t+1}}=F_{2t+3}N_{d_{t+1}-2},\] so \[s_b(\nu_3)-s_b(\nu)=F_{2t+3}\Tr(N_{d_{t+1}-2}B).\] Because $d_{t+1}=e_{t+1}-e_t-1\geq (m+2)-(m-1)-1=2$ by \eqref{Eq27}, there follows that $N_{d_{t+1}-2}$ is a matrix with nonnegative entries whose second diagonal entry is positive. By Lemma \ref{Lem7}, $s_b(\nu_3)-s_b(\nu)>0$, a contradiction. 
\par 
The fourth case we consider is that in which $a_1=0$, $a_2=1$, and $a_m=a_{m+1}=1$. By now the reader is probably familiar with the general pattern of the proof. We use the minimality of $m$ to say that $a_i=1$ for all positive even integers $i<m$ and $a_j=1$ for all positive odd integers $j<m$. The index $m$ must be even. Let $t=\dfrac{m-2}{2}$, and observe that $e_j=2j+2$ for all $j\in\{0,1,\ldots,t\}$ and that \[e_t=m<m+1=e_{t+1}<e_{t+2}<\cdots<e_\ell=k.\] We know that $m<r$ because Lemma \ref{Lem4} tells us that $a_j=0$ for all $j\in\{r+1,\ldots,k-1\}$. Therefore, $e_{t+1}\leq r<k=e_\ell$. Let $\nu_4=1+b+\displaystyle{\sum_{i=0}^{t-1}b^{e_i+1}+\sum_{i=t+1}^\ell b^{e_i}}$. We have \[\nu_4-1=b+\sum_{i=0}^{t-1}b^{e_i+1}+\sum_{i=t+1}^\ell b^{e_i}=b+\sum_{i=0}^{t-1}b^{e_{i+1}-1}+\sum_{i=t+1}^\ell b^{e_i}\] \[=b+\sum_{i=1}^tb^{e_i-1}+\sum_{i=t+1}^\ell b^{e_i}=\sum_{i=0}^tb^{e_i-1}+\sum_{i=t+1}^\ell b^{e_i}<\nu-1<b^k+\sum_{i=0}^yb^{r-2i},\] so $\nu_4\in I(k,r,y)$ and $\nu_4<\nu$. This implies that $s_b(\nu_4)<s_b(\nu)$. Let $B=M_{d_{t+2}}M_{d_{t+3}}\cdots M_{d_\ell}$. The matrix product defining $B$ is nonempty because $e_{t+1}<e_\ell$, and $B\in\Xi$. Consequently, \[B=\left(\begin{array}{ccc}
\xi_1 & \xi_2 & \xi_1 \\
\xi_1 & \xi_2 & \xi_1 \\
\xi_3 & \xi_4 & \xi_3 \end{array} \right)\] for some nonnegative real $\xi_1,\xi_2,\xi_3,\xi_4$  with $\xi_2>0$. Because $d_{t+1}=e_{t+1}-e_t-1=(m+1)-m-1=0$, we may use Theorem \ref{Thm2} to deduce that \[s_b(\nu)=\Tr(N_2M_1^tM_0B)\] and \[s_b(\nu_4)=\Tr(N_1M_1^{t+1}B).\] One may show that \[N_1M_1^{t+1}-N_2M_1^tM_0=F_{2t+2}\left(\begin{array}{ccc}
1 & -1 & 1 \\
1 & -1 & 1 \\
0 & 0 & 0 \end{array} \right),\] so \[s_b(\nu_4)-s_b(\nu)=\Tr((N_1M_1^{t+1}-N_2M_1^tM_0)B)=F_{2t+2}(\xi_3+\xi_4)\geq 0.\] This is our desired contradiction. 
\par 
Finally, we consider the case in which $a_1=a_2=0$. In this case, $\nu\equiv 1\pmod{b^3}$. Because $\nu-1<b^k+\displaystyle{\sum_{i=0}^yb^{r-2i}}$ and $\nu-1\equiv b^k+\displaystyle{\sum_{i=0}^yb^{r-2i}}\equiv 0\pmod b$, it follows that $\nu-1\leq b^k-b+\displaystyle{\sum_{i=0}^yb^{r-2i}}$. This means that $\nu+b\in I(k,r,y)$, so $s_b(\nu+b)\leq s_b(\nu)$ by the definition of $\nu$. As $\nu+b\equiv b+1\pmod{b^3}$, we find by \eqref{Eq1} and Lemma \ref{Lem2} that \[s_b(\nu+b)=s_b\left(\frac{(\nu+b)-1}{b}\right)+s_b\left(\frac{(\nu+b)+b^2-b-1}{b^2}\right)\] and \[s_b\left(\frac{\nu-1}{b^2}+1\right)=s_b\left(\frac{\left(\frac{\nu-1}{b^2}+1\right)-1}{b}\right)+s_b\left(\frac{\left(\frac{\nu-1}{b^2}+1\right)-1}{b}+1\right)\] \[>s_b\left(\frac{\left(\frac{\nu-1}{b^2}+1\right)-1}{b}\right)=s_b\left(\frac{\nu-1}{b^3}\right)=s_b\left(\frac{\nu-1}{b}\right).\] Therefore,
\[s_b(\nu+b)=s_b\left(\frac{(\nu+b)-1}{b}\right)+s_b\left(\frac{(\nu+b)+b^2-b-1}{b^2}\right)\] \[=s_b\left(\frac{\nu-1}{b}+1\right)+s_b\left(\frac{\nu-1}{b^2}+1\right)\] \[>s_b\left(\frac{\nu-1}{b}+1\right)+s_b\left(\frac{\nu-1}{b}\right)=s_b(\nu).\] With this contradiction, the proof is complete.   
\end{proof} 
\begin{lemma} \label{Lem10}
Preserving the notation from Lemma \ref{Lem9}, we have $e_{\ell-1-\lambda}=r-2\lambda-1$. 
\end{lemma} 
\begin{proof} 
The proof is very similar to those of the first and third cases considered in the proof of Lemma \ref{Lem9}. Suppose, by way of contradiction, that $e_{\ell-1-\lambda}\leq r-2\lambda-2$, and let $\nu'=1+\displaystyle{\sum_{i=0}^{\ell-1-\lambda}b^{e_i+1}+\sum_{i=\ell-\lambda}^\ell b^{e_i}}$ and $\nu''=\nu'+b$. We have \[\nu'-1<\nu''-1=b+\sum_{i=0}^{\ell-1-\lambda}b^{e_i+1}+\sum_{i=\ell-\lambda}^\ell b^{e_i}\leq\sum_{n=1}^{e_{\ell-1-\lambda}+1}b^n+\sum_{i=\ell-\lambda}^\ell b^{e_i}\] \[\leq\sum_{n=1}^{r-2\lambda-1}b^n+\sum_{i=\ell-\lambda}^\ell b^{e_i}<b^{r-2\lambda}+\sum_{i=\ell-\lambda}^\ell b^{e_i}=b^{r-2\lambda}+b^k+\sum_{i=0}^{\lambda-1}b^{r-2i}\] \[=b^k+\sum_{i=0}^{\lambda}b^{r-2i}\leq b^k+\sum_{i=0}^yb^{r-2i},\] where we convene to let $\displaystyle{\sum_{i=0}^{\lambda-1}b^{r-2i}=0}$ if $\lambda=0$. Therefore, $\nu',\nu''\in I(k,r,y)$. This means that $s_b(\nu'),s_b(\nu'')\leq s_b(\nu)$, so we will derive a contradiction by showing that $s_b(\nu')>s_b(\nu)$ or $s_b(\nu'')>s_b(\nu)$.
\par 
Let $B=M_{d_{\ell-\lambda+1}}M_{d_{\ell-\lambda+2}}\cdots M_{d_\ell}$, where we define $B$ to be the $3\times 3$ identity matrix if $\lambda=0$. By Lemma \ref{Lem6}, $B\in\Xi'$. To ease our notation, let $t=\ell-1-\lambda$ and $q=d_{\ell-\lambda}$. We have \[s_b(\nu)=\Tr(N_{e_0}M_1^tM_qB),\] \[s_b(\nu')=\Tr(N_{e_0+1}M_1^tM_{q-1}B),\] and \[s_b(\nu'')=\Tr(N_1M_{e_0-1}M_1^tM_{q-1}B).\] It is straightforward to show that $N_2M_1^tM_{q-1}-N_1M_1^tM_q=F_{2t+2}N_{q-2}$ and $N_1M_1^{t+1}M_{q-1}-N_2M_1^tM_q=F_{2t+3}N_{q-2}$. If $\lambda=0$, then it follows from our assumption that $e_{\ell-1-\lambda}\leq r-2\lambda-2$ that \[q=e_{\ell-\lambda}-e_{\ell-1-\lambda}-1=e_{\ell}-e_{\ell-1}-1=k-e_{\ell-1}-1\geq k-(r-2)-1\geq 3.\] If $\lambda>0$, then \[q=e_{\ell-\lambda}-e_{\ell-1-\lambda}-1=e_{\ell-1-(\lambda-1)}-e_{\ell-1-\lambda}-1=(r-2(\lambda-1))-e_{\ell-1-\lambda}-1\] \[\geq (r-2(\lambda-1))-(r-2\lambda-2)-1=3.\] In either case, $q\geq 3$, so both $F_{2t+2}N_{q-2}$ and $F_{2t+3}N_{q-2}$ are $3\times 3$ matrices with nonnegative real entries whose second diagonal entries are positive. It follows from Lemma \ref{Lem7} that if $e_0=1$, then \[s_b(\nu')-s_b(\nu)=\Tr((N_2M_1^tM_{q-1}-N_1M_1^tM_q)B)=\Tr(F_{2t+2}N_{q-2}B)>0.\] Similarly, if $e_0=2$, then \[s_b(\nu'')-s_b(\nu)=\Tr((N_1M_1^{t+1}M_{q-1}-N_2M_1^tM_q)B)=\Tr(F_{2t+3}N_{q-2}B)>0.\] This is a contradiction, so the proof is complete.  
\end{proof}
For nonnegative integers $k,r,y$ with $2y<r<k-1$, let $w=a_ka_{k-1}\cdots a_0$ be the ordinary base $b$ expansion of $\nu(k,r,y)-1$. We summarize here the information we have gained about this expansion. From Lemma \ref{Lem4}, we know that $a_{i}\in\{0,1\}$ for all $i$. We know from the same lemma that $a_k=1$, $a_0=0$, and $a_j=0$ for all $j\in\{r+1,\ldots,k-1\}$. Therefore, $w=10^{k-r-1}v$, where $v$ is a word of length $r+1$ over the binary alphabet $\{0,1\}$ that ends in the letter (digit) $0$. From Lemmas \ref{Lem9} and \ref{Lem10} and \eqref{Eq6}, we find that there exists some $\lambda\in\{0,1,\ldots,y\}$ such that 
\begin{enumerate}[A.]
\item{$a_{r-2i}=1$ and $a_{r-2i-1}=0$ for all $i\in\{0,1,\ldots,\lambda-1\}$.} 
\item{$a_{r-2\lambda}=0$.} 
\item{$a_{r-2\lambda-2i+1}=1$ and $a_{r-2\lambda-2i}=0$ for all $i\in\{1,2,\ldots,\left\lfloor r/2\right\rfloor-\lambda\}$.}
\end{enumerate} 
That is, \[v=\begin{cases} (10)^\lambda 0(10)^{\left\lfloor r/2\right\rfloor-\lambda}0, & \mbox{if } 2\nmid r; \\ (10)^\lambda 0(10)^{\left\lfloor r/2\right\rfloor-\lambda}, & \mbox{if } 2\vert r. \end{cases}\]
This means that 
\begin{equation} \label{Eq7} 
w=\begin{cases} 10^{k-r-1}(10)^\lambda 0(10)^{\left\lfloor r/2\right\rfloor-\lambda}0, & \mbox{if } 2\nmid r; \\ 10^{k-r-1}(10)^\lambda 0(10)^{\left\lfloor r/2\right\rfloor-\lambda}, & \mbox{if } 2\vert r. \end{cases}
\end{equation} This leads us to the following.
\begin{definition} \label{Def6}
Let $k,r,x$ be nonnegative integers such that $2x<r<k-1$. Define the word $w(k,r,x)$ by \[w(k,r,y)=\begin{cases} 10^{k-r-1}(10)^x0(10)^{\left\lfloor r/2\right\rfloor-x} 0, & \mbox{if } 2\nmid r; \\ 10^{k-r-1}(10)^x 0(10)^{\left\lfloor r/2\right\rfloor-x}, & \mbox{if } 2\vert r, \end{cases}\] and let $\gamma(k,r,x)$ be the positive integer whose ordinary base $b$ expansion is $w(k,r,x)$.
\end{definition}
\begin{definition} \label{Def7} 
For any integers $k,r,y$, define \[V(k,r,x)=\frac 15((k-r)(2L_{r+2}-L_{r-4x+1})+2L_{r+1}+L_{r-4x+2}).\] 
\end{definition} 
\begin{lemma} \label{Lem11} 
Let $k,r,x,y$ be nonnegative integers such that $2x\leq 2y<r<k-1$. We have \[\gamma(k,r,x)+1\in I(k,r,y)\] and \[s_b(\gamma(k,r,x)+1)=V(k,r,x).\] 
\end{lemma} 
\begin{proof} 
The ordinary base $b$ expansion of $b^k+\displaystyle{\sum_{i=0}^yb^{r-2i}}$ is \[10^{k-r-1}(10)^{y+1}0^{r-2y-1},\] so one may see from Definition \ref{Def6} and the hypothesis that $x\leq y$ that $\gamma(k,r,x)+1\leq b^k+\displaystyle{\sum_{i=0}^yb^{r-2i}}$. 
It should also be clear from Definition \ref{Def6} that $b^k<\gamma(k,r,x)+1$, so $\gamma(k,r,x)+1\in I(k,r,y)$. If we choose to write $\gamma(k,r,x)=\displaystyle{\sum_{i=0}^\ell b^{e_i}}$ for nonnegative integers $e_0,e_1,\ldots,e_\ell$ with $e_0<e_1\cdots<e_\ell$, then \[e_0=\begin{cases} 2, & \mbox{if } 2\nmid r; \\ 1, & \mbox{if } 2\vert r, \end{cases}\] \[e_j=e_0+2j\text{ for all }j\in\{1,2,\ldots,t\},\] \[e_{t+j}=r-2x+2j=e_0+2t+2j+1\text{ for all }j\in\{1,2,\ldots,x\},\] and \[e_\ell=k,\] where $t=\left\lfloor r/2\right\rfloor-x-1$. Observe that 
the letter (digit) $1$ appears exactly $\left\lfloor r/2\right\rfloor+1$ times in the word $w(k,r,x)$, so it follows from \eqref{Eq6} that $\ell=\left\lfloor r/2\right\rfloor$ (implying that $t+1=\ell-x$). Let us first assume $x=0$. Setting $d_j=e_j-e_{j-1}-1$ for all $j\in\{1,2,\ldots,\ell\}$, we have 
\[d_j=\begin{cases} 1, & \mbox{if } j\in\{1,2,\ldots,\ell-1\};\\ k-r, & \mbox{if } j=\ell. \end{cases}\] We may invoke Theorem \ref{Thm2} to see that 
\begin{equation} \label{Eq28}
s_b(\gamma(k,r,0)+1)=\Tr(N_{e_0}M_{d_1}M_{d_2}\cdots M_{d_\ell})=\Tr(N_{e_0}M_1^{\ell-1}M_{k-r}).
\end{equation} Let \[T=\left( \begin{array}{ccc}
F_r & F_{r-1} & F_r \\ 
F_r & F_{r-1} & F_r \\
0 & 0 & 0 \end{array} \right).\] If $2\vert r$, we may use Lemma \ref{Lem8} to find that \[N_{e_0}M_1^{\ell-1}=N_1M_1^{\frac 12(r-2)}=T.\] Similarly, if $2\nmid r$, then \[N_{e_0}M_1^{\ell-1}=N_2M_1^{\frac 12(r-3)}=T.\] Therefore, regardless of the parity of $r$, we have $s_b(\gamma(k,r,0)+1)=\Tr(TM_{k-r})$ by \eqref{Eq28}. Thus, \[s_b(\gamma(k,r,0)+1)=\Tr(TM_{k-r})=(k-r)F_{r+1}+F_{r+2},\] where the last equality is obtained by simply calculating the matrix $TM_{k-r}$ by hand. One may easily show that $F_{r+1}=\dfrac 15(2L_{r+2}-L_{r+1})$ and $F_{r+2}=\dfrac 15(2L_{r+1}+L_{r+2})$, so \[s_b(\gamma(k,r,0)+1)=\frac 15((k-r)(2L_{r+2}-L_{r+1})+2L_{r+1}+L_{r+2})=V(k,r,0).\] 
\par 
Let us now assume $x>0$. Again, we set $d_j=e_j-e_{j-1}-1$ for all $j\in\{1,2,\ldots,\ell\}$, and we see that \[d_j=\begin{cases} 1, & \mbox{if } j\in\{1,2,\ldots,\ell-1\}\setminus\{t+1\}; \\ 2, & \mbox{if } j=t+1; \\ k-r-1, & \mbox{if } j=\ell. \end{cases}\] By Theorem \ref{Thm2}, \[s_b(\gamma(k,r,x)+1)=\Tr(N_{e_0}M_{d_1}M_{d_2}\cdots M_{d_\ell})\] \[=\Tr(N_{e_0}M_1^tM_2M_1^{\ell-1-(t+1)}M_{k-r-1})=\Tr(N_{e_0}M_1^tM_2M_1^{x-1}M_{k-r-1}).\] 
If $2\vert r$, then Lemma \ref{Lem8} allows us to find that \[N_{e_0}M_1^t=N_1M_1^{r/2-x-1}=\left( \begin{array}{ccc}
F_{r-2x} & F_{r-2x-1} & F_{r-2x} \\ 
F_{r-2x} & F_{r-2x-1} & F_{r-2x} \\
0 & 0 & 0 \end{array} \right).\] Similarly, if $2\nmid r$, then \[N_{e_0}M_1^t=N_2M_1^{(r-1)/2-x-1}=\left( \begin{array}{ccc}
F_{r-2x} & F_{r-2x-1} & F_{r-2x} \\ 
F_{r-2x} & F_{r-2x-1} & F_{r-2x} \\
0 & 0 & 0 \end{array} \right).\] Therefore, no matter the parity of $r$, we have \[N_{e_0}M_1^tM_2M_1^{x-1}M_{k-r-1}=\left( \begin{array}{ccc}
F_{r-2x} & F_{r-2x-1} & F_{r-2x} \\ 
F_{r-2x} & F_{r-2x-1} & F_{r-2x} \\
0 & 0 & 0 \end{array} \right)M_2M_1^{x-1}M_{k-r-1}.\]
Using Lemma \ref{Lem8} again, we find that \[\left( \begin{array}{ccc}
F_{r-2x} & F_{r-2x-1} & F_{r-2x} \\ 
F_{r-2x} & F_{r-2x-1} & F_{r-2x} \\
0 & 0 & 0 \end{array} \right)M_2M_1^{x-1}M_{k-r-1}=\left( \begin{array}{ccc}
\chi_1 & \chi_2 & \chi_1 \\ 
\chi_1 & \chi_2 & \chi_1 \\
0 & 0 & 0 \end{array} \right),\] where \[\chi_1=(k-r)(F_{r-2x+3}F_{2x}+F_{r-2x+1}F_{2x-1})-F_{r-2x+3}F_{2x-2}-F_{r-2x+1}F_{2x-3}\] and \[\chi_2=F_{r-2x+3}F_{2x}+F_{r-2x+1}F_{2x-1}.\] Therefore, \[s_b(\gamma(k,r,x)+1)=\Tr(N_{e_0}M_1^tM_2M_1^{x-1} M_{k-r-1})=\chi_1+\chi_2\] \[=(k-r)(F_{r-2x+3}F_{2x}+F_{r-2x+1}F_{2x-1})+F_{r-2x+3}F_{2x-1}+F_{r-2x+1}F_{2x-2}.\] We now make use of the well-known identity \[F_mF_n=\frac 15(L_{m+n}-(-1)^nL_{m-n})\] to obtain \[F_{r-2x+3}F_{2x}+F_{r-2x+1}F_{2x-1}=\frac 15(L_{r+3}-L_{r-4x+3}+L_r+L_{r-4x+2})\] \[=\frac 15(L_{r+2}+L_{r+1}+L_r-L_{r-4x+1})=\frac 15(2L_{r+2}-L_{r-4x+1})\] and \[F_{r-2x+3}F_{2x-1}+F_{r-2x+1}F_{2x-2}=\frac 15(L_{r+2}+L_{r-4x+4}+L_{r-1}-L_{r-4x+3})\] \[=\frac 15(L_{r+1}+L_r+L_{r-1}+L_{r-4x+2})=\frac 15(2L_{r+1}+L_{r-4x+2}).\] Consequently, \[s_b(\gamma(k,r,x)+1)=(k-r)\left(\frac 15(2L_{r+2}-L_{r-4x+1})\right)+\frac 15(2L_{r+1}+L_{r-4x+2})\] \[=V(k,r,x).\] 
\end{proof}   
Before proceeding to our main result, we need one final easy lemma. 
\begin{lemma} \label{Lem12} 
For any integers $k,r,x$, \[V(k,r,x+1)-V(k,r,x)=\frac15((k-r)(L_{r-4x+1}-L_{r-4x-3})+L_{r-4x-2}-L_{r-4x+2}).\]
\end{lemma} 
\begin{proof} 
Referring to Definition \ref{Def7}, we have \[V(k,r,x+1)=A+\frac 15((k-r)(-L_{r-4(x+1)+1})+L_{r-4(x+1)+2})\] and \[V(k,r,x)=A+\frac 15((k-r)(-L_{r-4x+1})+L_{r-4x+2}),\] where $A=\dfrac15((k-r)(2L_{r+2})+2L_{r+1})$. Thus, 
\[V(k,r,x+1)-V(k,r,x)=\frac15((k-r)(L_{r-4x+1}-L_{r-4x-3})+L_{r-4x-2}-L_{r-4x+2}).\]
\end{proof} 
We are finally in a position to prove our main result. 
\begin{theorem} \label{Thm3} 
Let $k,r,y$ be nonnegatve integers such that $2y<r<k-1$. If $r$ is even, then \[\nu(k,r,y)=\gamma(k,r,y)+1\text{ and }\mu(k,r,y)=V(k,r,y).\] If $k-2=r\equiv3\pmod 4$ and $y\geq\dfrac{r+1}{4}$, then \[\nu(k,r,y)=\gamma(k,r,(r-3)/4)+1\text{ and }\mu(k,r,y)=V(k,r,(r-3)/4).\] 
Otherwise, \[\nu(k,r,y)=\gamma(k,r,\delta(r,y))+1\text{ and }\mu(k,r,y)=V(k,r,\delta(r,y)),\] where $\delta(r,y)=\min\left\{\left\lceil\frac{r-1}{4}\right\rceil,y\right\}$.
\end{theorem} 
\begin{proof} 
We saw in the discussion immediately preceding Definition \ref{Def6} that $\nu(k,r,y)=\gamma(k,r,\lambda)+1$ for some $\lambda\in\{0,1,\ldots,y\}$; we simply need to determine the value of $\lambda$. By Definition \ref{Def3}, $\lambda$ is the element of $\{0,1,\ldots,y\}$ such that $s_b(\gamma(k,r,\lambda)+1)$ is maximized and $\gamma(k,r,\lambda)+1\in I(k,r,y)$. Lemma \ref{Lem11} tells us that $\gamma(k,r,x)+1\in I(k,r,y)$ and $s_b(\gamma(k,r,x)+1)=V(k,r,x)$ for all $x\in\{0,1,\ldots,y\}$, so \[\mu(k,r,y)=\max\{V(k,r,x)\colon x\in\{0,1,\ldots,y\}\}.\] That is, $\lambda$ must be the element of the set $\{0,1,\ldots,y\}$ that maximizes $V(k,r,\lambda)$ (if there are multiple such elements, we choose the smallest in accordance with the definition of $\nu(k,r,y)$ in Definition \ref{Def3}).  
We first assume $r$ is even. We wish to show that $\lambda=y$. This is immediate if $y=0$, so we will assume $y>0$. Choose some $x\in\{0,1,\ldots,y-1\}$.  
We make use of an easily-proven fact about Lucas numbers that states that if $m,n$ are odd integers (not necessarily positive) with $m<n$, then $L_m<L_n$. In this case, we use the assumption that $r$ is even to see that $r-4x-3$ and $r-4x+1$ are odd integers with $r-4x-3<r-4x+1$. Therefore, $L_{r-4x+1}-L_{r-4x-3}>0$. Because $k-r\geq 2$ by hypothesis, we may use Lemma \ref{Lem12} to find that \[V(k,r,x+1)-V(k,r,x)\geq \frac 15(2(L_{r-4x+1}-L_{r-4x-3})+L_{r-4x-2}-L_{r-4x+2})\] \[=\frac 15(L_{r-4x-1}-L_{r-4x-5}).\] As $r-4x-5$ and $r-4x-1$ are odd integers with $r-4x-5<r-4x-1$, we see that \[V(k,r,x+1)-V(k,r,x)>0.\] Because $x$ was arbitrary, this shows that \[V(k,r,0)<V(k,r,1)<\cdots<V(k,r,y),\] so $\lambda=y$. 
\par 
We now assume $r$ is odd. Let $x_1,x_2$ be integers such that $x_1\leq\dfrac{r-5}{4}<\dfrac{r-1}{4}\leq x_2$. We will make use of the fact that $L_m=L_{-m}$ for any even integer $m$. Because $x_1\leq\dfrac{r-5}{4}$, we obtain the inequalities $r-4x_1+1>r-4x_1-3>0$. Therefore, $L_{r-4x_1+1}-L_{r-4x_1-3}>0$. Because $k-r\geq 2$, we see from Lemma \ref{Lem12} that \[V(k,r,x_1+1)-V(k,r,x_1)\geq \frac 15(2(L_{r-4x_1+1}-L_{r-4x_1-3})+L_{r-4x_1-2}-L_{r-4x_1+2})\] \[=\frac 15(L_{r-4x_1-1}-L_{r-4x_1-5}).\] We know that $L_{r-4x_1-1}-L_{r-4x_1-5}>0$ because $r-4x_1-5\geq0$. Thus,
\begin{equation} \label{Eq8} 
V(k,r,x_1+1)>V(k,r,x_1). 
\end{equation} 
If $y<\left\lceil\frac{r-1}{4}\right\rceil$, then every element of the set $\{0,1,\ldots,y-1\}$ is less than $\dfrac{r-5}{4}$, so \eqref{Eq8} shows that \[V(k,r,0)<V(k,r,1)<\cdots<V(k,r,y).\] This shows that if $y<\left\lceil\frac{r-1}{4}\right\rceil$, then \[\lambda=y=\delta(r,y),\] which agrees with the statement of the theorem. Therefore, we will henceforth assume that $y\geq\left\lceil\frac{r-1}{4}\right\rceil$ so that $\delta(r,y)=\left\lceil\frac{r-1}{4}\right\rceil$. 
It follows from the inequality $x_2\geq\dfrac{r-1}{4}$ that $4x_2-r+3>4x_2-r-1\geq -2$. If $4x_2-r-1\geq 0$, then \[L_{4x_2-r-1}-L_{4x_2-r+3}<0.\] If $4x_2-r-1=-2$, then \[L_{4x_2-r-1}-L_{4x_2-r+3}=L_{-2}-L_2=0.\] Hence, \[L_{4x_2-r-1}-L_{4x_2-r+3}\leq 0.\] Because $r-4x_2+1$ and $r-4x_2-3$ are even, \[L_{r-4x_2+1}-L_{r-4x_2-3}=L_{4x_2-r-1}-L_{4x_2-r+3}\leq 0.\] Using Lemma \ref{Lem12} and the fact that $k-r\geq 2$, we have \[V(k,r,x_2+1)-V(k,r,x_2)\leq\frac 15(2(L_{r-4x_2+1}-L_{r-4x_2-3})+L_{r-4x_2-2}-L_{r-4x_2+2})\] \[=\frac15(L_{r-4x_2-1}-L_{r-4x_2-5}).\] Because $4x_2-r+5>4x_2-r+1\geq 0$ and $r-4x_2-1,r-4x_2-5$ are even, we have \[L_{r-4x_2-1}-L_{r-4x_2-5}=L_{4x_2-r+1}-L_{4x_2-r+5}<0.\] This shows that 
\begin{equation} \label{Eq9} 
V(k,r,x_2+1)<V(k,r,x_2). 
\end{equation}  
We now have three short cases to consider. In the first case, assume $r\equiv 1\pmod 4$.
We have shown through \eqref{Eq8} and \eqref{Eq9} that \[V(k,r,m)>V(k,r,m-1)>V(k,r,m-2)>\cdots\] and \[V(k,r,m)>V(k,r,m+1)>V(k,r,m+2)>\cdots,\] where $m=\displaystyle{\frac{r-1}{4}}$. This implies that $\dfrac{r-1}{4}$ is the unique value of $x$ that maximizes $V(k,r,x)$. Recall that we have assumed that $y\geq\left\lceil\frac{r-1}{4}\right\rceil$, so $\dfrac{r-1}{4}\in\{0,1,\ldots,y\}$. Consequently, \[\lambda=\frac{r-1}{4}=\delta(r,y).\]
\par 
For the second case, we suppose that $r\equiv 3\pmod 4$ and $r<k-2$. We need to show that $\lambda=\delta(r,y)$. We may use \eqref{Eq8} and \eqref{Eq9} to see that \[V(k,r,n)>V(k,r,n-1)>V(k,r,n-2)>\cdots\] and \[V(k,r,n+1)>V(k,r,n+2)>V(k,r,n+3)>\cdots,\] where $n=\dfrac{r-3}{4}$. By Lemma \ref{Lem12}, \[V(k,r,n+1)-V(k,r,n)=\frac15((k-r)(L_4-L_0)+L_1-L_5)\] \[=\frac15((k-r)(7-2)+1-11)=k-r-2>0,\] so $V(k,r,n)<V(k,r,n+1)$. This shows that $n+1=\dfrac{r+1}{4}$ is the unique value of $x$ that maximizes $V(k,r,x)$. Because $y\geq\left\lceil\frac{r-1}{4}\right\rceil=\frac{r+1}{4}$, it follows that $\frac{r+1}{4}\in\{0,1,\ldots,y\}$. Therefore, \[\lambda=\tfrac{r+1}{4}=\left\lceil\tfrac{r-1}{4}\right\rceil=\delta(r,y).\] 
\par 
For the third and final case, we assume $k-2=r\equiv 3\pmod 4$. We wish to show that $\lambda=\dfrac{r-3}{4}$. Using \eqref{Eq8} and \eqref{Eq9}, we obtain \[V(k,r,n)>V(k,r,n-1)>V(k,r,n-2)>\cdots\] and \[V(k,r,n+1)>V(k,r,n+2)>V(k,r,n+3)>\cdots,\] where $n=\dfrac{r-3}{4}$. Lemma \ref{Lem12} tells us that \[V(k,r,n+1)-V(k,r,n)=\frac15((k-r)(L_4-L_0)+L_1-L_5)\] \[=\frac15((k-r)(7-2)+1-11)=k-r-2=0,\] so $V(k,r,n)=V(k,r,n+1)$. Therefore, $V(k,r,x)$ obtains its maximum when $x=n$ and when $x=n+1$. This means that $\mu(k,r,y)=V(k,r,n)$. Either $\nu(k,r,y)=\gamma(k,r,n)+1$ or $\gamma(k,r,n+1)+1$. Note that $\gamma(k,r,n)+1<\gamma(k,r,n+1)+1$ by Definition \ref{Def6}. Since $\nu(k,r,y)$ is defined to be the \emph{smallest} element of $I(k,r,y)$ such that $s_b(\nu(k,r,y)+1)=\mu(k,r,y)$, it follows that $\nu(k,r,y)=\gamma(k,r,n)+1$. That is, $\lambda=n=\dfrac{r-3}{4}$. 
\end{proof}
We end this section with a definition and a theorem that will prove useful in the next section. 
\begin{definition} \label{Def4}
For any integers $k$ and $m$ with $2\leq m\leq k$, let $G_k(m)=b^k+h_m$ and $H_k(m)=F_{m+2}+(k-m)F_m$. 
\end{definition}  
Although it is possible to give an easy inductive proof of the following theorem using the simple Fibonacci-like recurrence \begin{equation} \label{Eq3}
H_k(m)=H_{k-1}(m-1)+H_{k-2}(m-2),
\end{equation} we prefer a proof based on Theorem \ref{Thm2}.
\par 
\begin{theorem} \label{Thm1}
For any integers $k$ and $m$ with $2\leq m\leq k$, we have \[s_b(G_k(m))=H_k(m).\]  
\end{theorem}
\begin{proof} 
Referring to Definition \ref{Def2}, we see that we may write \[G_k(m)-1=b^k+\sum_{i=0}^{\lfloor\frac{m-3}{2}\rfloor}b^{m-2-2i}=\sum_{i=0}^\ell b^{e_i},\] where \[e_0=m-2-2\left\lfloor\tfrac{m-3}{2}\right\rfloor=\begin{cases} 2, & \mbox{if } 2\vert m; \\ 1, & \mbox{if } 2\nmid m, \end{cases}\] $\ell=\left\lfloor\frac{m-1}{2}\right\rfloor$, $e_\ell=k$, and $e_j=e_0+2j$ for all $j\in\{0,1,\ldots,\ell-1\}$. For example, \[G_{12}(7)-1=b^{12}+\sum_{i=0}^{\left\lfloor\frac{7-3}{2}\right\rfloor}b^{7-2-2i}=b^1+b^3+b^5+b^{12}.\] Defining $d_i$ as in Theorem \ref{Thm2}, we have $d_j=1$ for all $j\in\{1,2,\ldots,\ell-1\}$ and \[d_\ell=e_\ell-e_{\ell-1}-1=k-(e_0+2(\ell-1))-1\] \[=k-\left(\left(m-2-2\left\lfloor\tfrac{m-3}{2}\right\rfloor\right)+2\left(\left\lfloor\tfrac{m-1}{2}\right\rfloor-1\right)\right)-1=k-m+1.\] By Theorem \ref{Thm2}, \[s_b(G_k(m))=\Tr(N_{e_0}M_1^{\ell-1}M_{k-m+1}).\] Using Lemma \ref{Lem8}, one may show that \[N_{e_0}M_1^{\ell-1}M_{k-m+1}=\left( \begin{array}{ccc}
\rho_1 & \rho_2 & \rho_1 \\ 
\rho_1 & \rho_2 & \rho_1 \\
0 & 0 & 0 \end{array} \right),\] where \[\rho_1=(k-m-1)(e_0F_{2\ell}+F_{2\ell-1})+e_0F_{2\ell+2}+F_{2\ell+1}\] and \[\rho_2=e_0F_{2\ell}+F_{2\ell-1}.\] Therefore, \[s_b(G_k(m))=\rho_1+\rho_2=(k-m)(e_0F_{2\ell}+F_{2\ell-1})+e_0F_{2\ell+2}+F_{2\ell+1}.\] If $m$ is even, then $e_0=2$ and $\ell=\dfrac{m-2}{2}$, so \[s_b(G_k(m))=(k-m)(2F_{m-2}+F_{m-3})+2F_m+F_{m-1}\] \[=(k-m)F_m+F_{m+2}=H_k(m).\] If $m$ is odd, then $e_0=1$ and $\ell=\dfrac{m-1}{2}$, so \[s_b(G_k(m))=(k-m)(F_{m-1}+F_{m-2})+F_{m+1}+F_m\] \[=(k-m)F_m+F_{m+2}=H_k(m).\] 
\end{proof}
\section{More Manageable Bounds} 
In the previous section, we derived fairly strong upper bounds (relative to those previously known) for the values of $s_b(n)$ for those integers $n$ satisfying $b^k\leq n\leq b^k+h_k$ for some integer $k\geq 3$. Unfortunately, these bounds are somewhat cumbersome. In this section, we will weaken them in order to make them cleaner and more easily applicable. 
\par
Suppose $k,m$, and $n$ are integers such that $2\leq m<k$ and $G_k(m)<n<G_k(m+1)$. We wish to use Theorem \ref{Thm3} to show that the point $(n,s_b(n))$ lies below the line segment connecting the points $(G_k(m),H_k(m))$ and \\ $(G_k(m+1),H_k(m+1))$. This leads us to the following definition. 
\begin{definition} \label{Def8}
For any integers $k$ and $m$ with $2\leq m<k$, define a function $J_{k,m}\colon\mathbb R\rightarrow\mathbb R$ by \[J_{k,m}(x)=\frac{H_k(m+1)-H_k(m)}{G_k(m+1)-G_k(m)}(x-G_k(m))+H_k(m).\]  
\end{definition} 
The graph of $J_{k,m}(x)$ is the line passing through $(G_k(m),H_k(m))$ and $(G_k(m+1),H_k(m+1))$. Our aim is to show that $s_b(n)<J_{k,m}(n)$ for all $n\in\{G_k(m)+1,\ldots,G_k(m+1)-1\}$. Observe that we may rewrite $J_{k,m}(x)$ as \[J_{k,m}(x)=\frac{F_{m-1}(k-m)}{h_{m+1}-h_m}(x-G_k(m))+H_k(m)\] because $H_k(m+1)-H_k(m)=(k-m-1)F_{m+1}+F_{m+3}-(k-m)F_m-F_{m+2}=F_{m-1}(k-m)$ and $G_k(m+1)-G_k(m)=h_{m+1}-h_m$. 
\par
\begin{lemma} \label{Lem15}
If $m$ and $x$ are positive integers with $m\geq 4$ and \\ $x<G_{m-1}(m-3)$, then $s_b(x)\leq F_m$.
\end{lemma} 
\begin{proof} 
Let $m$ and $x$ be as in the statement of the lemma. If $x<b^{m-1}$, then it follows from Proposition \ref{Prop1} that $s_b(x)\leq F_m$. If $x=b^{m-1}$, then $s_b(x)=1<F_m$ by Lemma \ref{Lem1}. Therefore, we will assume $b^{m-1}<x$. Note that the inequalities $b^{m-1}<x<G_{m-1}(m-3)$ force $m\geq 6$. Since \[b^{m-1}<x\leq G_{m-1}(m-3)-1=b^{m-1}+\sum_{i=0}^{\left\lfloor\frac{m-6}{2}\right\rfloor}b^{m-5-2i},\] it follows from Definition \ref{Def3} that $x\in I\left(m-1,m-5,\left\lfloor\frac{m-6}{2}\right\rfloor\right)$. Referring to Definition \ref{Def3} again, we find that $s_b(x)\leq\mu\left(m-1,m-5,\left\lfloor\frac{m-6}{2}\right\rfloor\right)$. Therefore, we simply need to show that 
\begin{equation} \label{Eq29} 
\mu\left(m-1,m-5,\left\lfloor\tfrac{m-6}{2}\right\rfloor\right)\leq F_m.
\end{equation}
Suppose $m$ is odd. It follows from Theorem \ref{Thm3} and Definition \ref{Def7} that \[\mu\left(m-1,m-5,\left\lfloor\tfrac{m-6}{2}\right\rfloor\right)=\mu\left(m-1,m-5,\tfrac{m-7}{2}\right)=V\left(m-1,m-5,\tfrac{m-7}{2}\right)\] \[=\frac15(4(2L_{m-3}-L_{10-m})+2L_{m-4}+L_{11-m}).\] Because $m$ is odd, $L_{10-m}=-L_{m-10}$ and $L_{11-m}=L_{m-11}$. Therefore, \[\mu\left(m-1,m-5,\left\lfloor\tfrac{m-6}{2}\right\rfloor\right)=\frac15(4(2L_{m-3}+L_{m-10})+2L_{m-4}+L_{m-11}).\] For any integer $t$, let \[\eta_t=\frac15(4(2L_{t-3}+L_{t-10})+2L_{t-4}+L_{t-11})\] and \[\kappa_t=11F_{t-5}.\] It is easy to see that the recurrence relations $\eta_{t+4}=3\eta_{t+2}-\eta_t$ and $\kappa_{t+4}=3\kappa_{t+2}-\kappa_t$ hold for all integers $t$. In addition, one may verify that $\kappa_5=\eta_5=0$ and $\kappa_7=\eta_7=11$. This implies that $\kappa_t=\eta_t$ for all odd integers $t$. In particular, \[\mu\left(m-1,m-5,\left\lfloor\tfrac{m-6}{2}\right\rfloor\right)=\eta_m=\kappa_m=11F_{m-5}<F_m,\] where the inequality $11F_{m-5}<F_m$ is easily-proven for all odd $m>6$ (for example, by showing that $11F_{m-5}=22F_{m-7}+11F_{m-8}<21F_{m-7}+13F_{m-8}=F_m$). This proves \eqref{Eq29} when $m$ is odd. 
\par 
Next, assume $m$ is even. By Theorem \ref{Thm3}, \[\mu\left(m-1,m-5,\left\lfloor\tfrac{m-6}{2}\right\rfloor\right)=V\left(m-1,m-5,\delta\left(m-5,\left\lfloor\tfrac{m-6}{2}\right\rfloor\right)\right),\] 
where \[\delta\left(m-5,\left\lfloor\tfrac{m-6}{2}\right\rfloor\right)=\min\left\{\left\lceil\tfrac{(m-5)-1}{4}\right\rceil,\left\lfloor\tfrac{m-6}{2}\right\rfloor\right\}=\left\lceil\tfrac{m-6}{4}\right\rceil.\] If $m\equiv 0\pmod 4$, 
then $\left\lceil\frac{m-6}{4}\right\rceil=\frac{m-4}{4}$, so \[\mu\left(m-1,m-5,\left\lfloor\tfrac{m-6}{2}\right\rfloor\right)=V\left(m-1,m-5,\tfrac{m-4}{4}\right)\] \[=\frac15(4(2L_{m-3}-L_0)+2L_{m-4}+L_1)=\frac15(8L_{m-3}+2L_{m-4}-7)\] \[<\frac15(8L_{m-3}+2L_{m-4}).\] Similarly, if $m\equiv 2\pmod 4$, then $\left\lceil\frac{m-6}{4}\right\rceil=\frac{m-6}{4}$, so \[\mu\left(m-1,m-5,\left\lfloor\tfrac{m-6}{2}\right\rfloor\right)=V\left(m-1,m-5,\tfrac{m-6}{4}\right)\] \[=\frac15(4(2L_{m-3}-L_2)+2L_{m-4}+L_3)=\frac15(8L_{m-3}+2L_{m-4}-8)\] \[<\frac15(8L_{m-3}+2L_{m-4}).\] In either case, we may 
use the identity $F_u=\dfrac15(L_{u+1}+L_{u-1})$ to see that \[\mu\left(m-1,m-5,\left\lfloor\tfrac{m-6}{2}\right\rfloor\right)<\frac15(8L_{m-3}+2L_{m-4})=\frac15(6L_{m-3}+2L_{m-2})\] \[=\frac15(4L_{m-3}+2L_{m-1})=2\cdot\frac15(L_{m-3}+L_{m-1})+\frac15(L_{m-3}+L_{m-4}+L_{m-5})\] \[=2F_{m-2}+\frac15(L_{m-2}+L_{m-5})<2F_{m-2}+\frac15(L_{m-2}+L_{m-4})\] \[=2F_{m-2}+F_{m-3}=F_m.\] This proves \eqref{Eq29} in the case when $m$ is even, so the proof is complete. 
\end{proof}  
Throughout the proofs of the following three lemmas, we assume $b=2$. For example, since $G_k(4)=b^k+h_4=b^k+b^2+1$, it will be understood that $G_k(4)=2^k+5$. 
\begin{lemma} \label{Lem18} 
For any positive integers $t$ and $x$ with $t\geq \left\lfloor\log_2x\right\rfloor+2$, we have \[s_2(2^t+x)\leq s_2(x)(t+1-\left\lfloor\log_2x\right\rfloor).\] 
\end{lemma} 
\begin{proof} 
The proof is by induction on $t$. The inequality $t\geq\lfloor\log_2x\rfloor+2$ forces $t\geq 2$. If $t=2$, then the inequality $t\geq\lfloor\log_2x\rfloor+2$ forces $x=1$ so that \[s_2(2^t+x)=s_2(5)=3=s_2(1)(2+1-\lfloor\log_21\rfloor)=s_2(x)(t+1-\lfloor\log_2x\rfloor).\] Now, suppose $t\geq 3$. If $x$ is even, then it follows from \eqref{Eq1} and induction on $t$ that \[s_2(2^t+x)=s_2(2^{t-1}+x/2)\leq s_2(x/2)((t-1)+1-\lfloor\log_2(x/2)\rfloor)\] \[=s_2(x)(t+1-\lfloor\log_2x\rfloor).\] Therefore, we may assume $x$ is odd. By induction on $t$, we have \[s_2(2^{t-1}+(x-1)/2)\leq s_2((x-1)/2)((t-1)+1-\left\lfloor\log_2\left((x-1)/2\right)\right\rfloor)\] \[=s_2((x-1)/2)(t+1-\lfloor\log_2x\rfloor)\] and \[s_2(2^{t-1}+(x+1)/2)\leq s_2((x+1)/2)((t-1)+1-\left\lfloor\log_2\left((x+1)/2\right)\right\rfloor)\] \[\leq s_2((x+1)/2)(t+1-\lfloor\log_2x\rfloor).\] Therefore, it follows from \eqref{Eq1} that \[s_2(2^t+x)=s_2\left(2^{t-1}+(x-1)/2\right)+s_2\left(2^{t-1}+(x+1)/2\right)\] \[\leq (s_2\left((x-1)/2\right)+s_2((x+1)/2))\left(t+1-\left\lfloor\log_2x\right\rfloor\right)\] \[=s_2(x)(t+1-\lfloor\log_2x\rfloor).\]
\end{proof} 
\begin{lemma} \label{Lem19} 
If $k,m,$ and $n$ are integers such that $4\leq m<k$ and \[2^k+2^{m-1}\leq n<2^k+G_{m-1}(m-3),\] then $s_2(n)<H_k(m)$.
\end{lemma} 
\begin{proof} 
Let $k,m,$ and $n$ be as stated in the lemma. Let $x=n-2^k$, and note that $\left\lfloor\log_2x\right\rfloor=m-1$. We may combine Lemmas \ref{Lem15} and \ref{Lem18} to get \[s_2(n)=s_2(2^k+x)\leq s_2(x)(k+1-\left\lfloor\log_2x\right\rfloor)=s_2(x)(k-m+2)\] \[\leq F_m(k-m+2)<F_m(k-m)+F_{m+2}=H_k(m).\]
\end{proof} 
\begin{lemma} \label{Lem20} 
If $k,m,$ and $n$ are integers such that $m\in\{5,7\}$, $m<k$, and $n=2^k+G_{m-1}(m-3)$, then \[s_2(n)<J_{k,m}(n).\]
\end{lemma} 
\begin{proof} 
The proof is very straightforward. We make use of the fact that $s_2(17)=5$ and $s_2(69)=14$. Suppose $m=5$. Then $k\geq 6$ and $n=2^k+G_{4}(2)=2^k+2^4+h_2=2^k+17$ by hypothesis. Using Lemma \ref{Lem18} and the remark following Definition \ref{Def8}, we have \[s_2(n)=s_2(2^k+17)\leq s_2(17)(k+1-\left\lfloor\log_217\right\rfloor)=5(k-3)<\frac{34}{5}(k-5)+13\] \[=\frac{3(k-5)}{21-11}((2^k+17)-(2^k+11))+13+5(k-5)\] \[=\frac{F_4(k-5)}{h_6-h_5}(n-G_k(5))+F_7+F_5(k-5)=J_{k,5}(n).\]
Next, suppose $m=7$. Then $k\geq 8$ and $n=2^k+G_6(4)=2^k+2^6+2^2+1=2^k+69$. Using Lemma \ref{Lem18} and the paragraph following Definition \ref{Def8} once again yields \[s_2(n)=s_2(2^k+69)\leq s_2(69)(k+1-\left\lfloor\log_269\right\rfloor)=14(k-5)\] \[<\frac{104}{21}(k-7)+34+13(k-7)\] \[=\frac{8(k-7)}{85-43}((2^k+69)-(2^k+43))+34+13(k-7)\] \[=\frac{F_6(k-7)}{h_8-h_7}(n-G_k(7))+F_9+F_7(k-7)=J_{k,7}(n).\] 
\end{proof} 
We now return to our base assumption that $b$ may represent any integer larger than $1$.
\begin{lemma} \label{Lem13} 
Let $m$ and $y$ be integers such that $m\geq 4$ and $1\leq y\leq\left\lfloor\frac{m-2}{2}\right\rfloor$. We have $L_{m-4y-1}\leq L_{m-5}$ and $L_{m-4y}\geq -L_{m-6}$. 
\end{lemma}  
\begin{proof} 
First, suppose $m$ is even. We use the fact that $L_u<L_v$ for any odd integers $u,v$ with $u<v$ to see that the maximum possible value of $L_{m-4y-1}$ is obtained when $y=1$. That is, $L_{m-4y-1}\leq L_{m-5}$. It is easy to see that $L_u\geq 2$ for all even integers $u$, so $L_{m-4y}\geq 2>-2\geq -L_{m-6}$. 
\par 
Next, suppose $m$ is odd. We use the fact that $L_{-u}=L_u<L_v=L_{-v}$ for any nonnegative even integers $u,v$ with $u<v$ to find that the bilateral  sequence $\ldots,L_{-4},L_{-2},L_0,L_2,L_4,\ldots$ is strictly convex. Therefore, the maximum possible value of $L_{m-4y-1}$ occurs when $y=1$ or when $y=\left\lfloor\frac{m-2}{2}\right\rfloor=\frac{m-3}{2}$. Because $L_{m-4(1)-1}=L_{m-5}=L_{5-m}=L_{m-4((m-3)/2)-1}$, the maximum possible value of $L_{m-4y-1}$ is $L_{m-5}$. As $L_u<L_v$ for any odd integers $u,v$ with $u<v$, the smallest possible value of $L_{m-4y}$ occurs when $y=\frac{m-3}{2}$. That is, $L_{m-4y}\geq L_{m-4((m-3)/2)}=L_{6-m}$. Now, $L_{-u}=-L_u$ for all odd integers $u$, so $L_{m-4y}\geq L_{6-m}=-L_{m-6}$.   
\end{proof} 
\begin{lemma} \label{Lem14}
If $m$ and $y$ are integers such that $m\geq 4$ and $1\leq y\leq\left\lfloor\frac{m-2}{2}\right\rfloor$, then \[\frac{L_{m-2}+L_{m-4y}}{5F_{m-1}}>\sqrt5-2.\] Furthermore, if $m\geq 4$ and $m\not\in\{5,7\}$, then \[\frac{L_{m-2}+L_{m-4}}{5F_{m-1}}\geq \frac{8}{21}.\]
\end{lemma}
\begin{proof} 
We will use the closed-form expressions $F_n=\dfrac{\phi^n-\overline\phi^n}{\sqrt5}$ and $L_n=\phi^n+\overline\phi^n$ for the Fibonacci and Lucas numbers. We also make use of the fact that $L_u\geq 2$ for any even integer $u$. For any integer $j$, we have \[\frac{L_{m-j}}{5F_{m-1}}-\frac{1}{\phi^{j-1}\sqrt5}=\frac{\phi^{m-j}+\overline\phi^{m-j}}{\left(\phi^{m-1}-\overline\phi^{m-1}\right)\sqrt5}-\frac{1}{\phi^{j-1}\sqrt5}\] \[=\frac{1}{\sqrt5}\frac{\overline\phi^{m-j}\left(\phi^{j-1}+\overline\phi^{j-1}\right)}{\phi^{j-1}\left(\phi^{m-1}-\overline\phi^{m-1}\right)}=\frac{1}{5\phi^{j-1}}\frac{\overline\phi^{m-j}L_{j-1}\sqrt5}{\phi^{m-1}-\overline\phi^{m-1}}=\frac{L_{j-1}}{5\phi^{j-1}F_{m-1}}\overline\phi^{m-j}\] \[=(-1)^{m-j}\frac{L_{j-1}}{5\phi^{m-1}F_{m-1}},\] where the last equality follows from the identity $\overline\phi=\dfrac{-1}{\phi}$. Therefore,  
\begin{equation} \label{Eq11} 
\frac{L_{m-j}}{5F_{m-1}}=\frac{1}{\phi^{j-1}\sqrt5}+(-1)^{m-j}\frac{L_{j-1}}{5\phi^{m-1}F_{m-1}}.
\end{equation} 
If $m$ is even, then we may set $j=2$ in \eqref{Eq11} and invoke the inequality $L_{m-4y}\geq 2$ (since $m-4y$ is even) to find that \[\frac{L_{m-2}+L_{m-4y}}{5F_{m-1}}\geq\frac{L_{m-2}+2}{5F_{m-1}}>\frac{L_{m-2}}{5F_{m-1}}=\frac{1}{\phi\sqrt5}+(-1)^{m-2}\frac{L_1}{5\phi^{m-1}F_{m-1}}\] \[=\frac{1}{\phi\sqrt5}+\frac{1}{5\phi^{m-1}F_{m-1}}>\frac{1}{\phi\sqrt5}>\sqrt5-2.\] Let us now suppose $m$ is odd. By Lemma \ref{Lem13}, \[\frac{L_{m-2}+L_{m-4y}}{5F_{m-1}}\geq\frac{L_{m-2}-L_{m-6}}{5F_{m-1}}=\frac{L_{m-3}+L_{m-5}}{5F_{m-1}}.\] 
Using \eqref{Eq11}, once with $j=3$ and once with $j=5$, yields \[\frac{L_{m-3}+L_{m-5}}{5F_{m-1}}=\frac{1}{\phi^2\sqrt5}+(-1)^{m-3}\frac{L_2}{5\phi^{m-1}F_{m-1}}+\frac{1}{\phi^4\sqrt5}+(-1)^{m-5}\frac{L_4}{5\phi^{m-1}F_{m-1}}\] \[>\frac{1}{\phi^2\sqrt5}+\frac{1}{\phi^4\sqrt5}=\sqrt5-2\] because $(-1)^{m-3}\dfrac{L_2}{5\phi^{m-1}F_{m-1}}$ and $(-1)^{m-5}\dfrac{L_4}{5\phi^{m-1}F_{m-1}}$ are positive. This completes the proof of the first inequality in the statement of the lemma. 
\par 
Assume, now, that $m\geq 4$ and $m\not\in\{5,7\}$. We may use \eqref{Eq11}, once with $j=2$ and again with $j=4$, to see that \[\frac{L_{m-2}+L_{m-4}}{5F_{m-1}}=\frac{1}{\phi\sqrt5}+(-1)^{m-2}\frac{L_1}{5\phi^{m-1}F_{m-1}}+\frac{1}{\phi^3\sqrt5}+(-1)^{m-4}\frac{L_3}{5\phi^{m-1}F_{m-1}}\] \[=\frac{1}{\phi\sqrt5}+\frac{1}{\phi^3\sqrt5}+\frac{(-1)^m(L_1+L_3)}{5\phi^{m-1}F_{m-1}}=\frac{1}{\phi\sqrt5}+\frac{1}{\phi^3\sqrt5}+\frac{(-1)^m}{\phi^{m-1}F_{m-1}}.\] For $m\geq 4$ and $m\not\in\{5,7\}$, it is easy to see that $\dfrac{(-1)^m}{\phi^{m-1}F_{m-1}}$ attains its minimum when $m=9$. Thus, \[\frac{L_{m-2}+L_{m-4}}{5F_{m-1}}\geq\frac{1}{\phi\sqrt5}+\frac{1}{\phi^3\sqrt5}+\frac{(-1)^9}{\phi^8F_8}=\frac{8}{21},\] where this last equality is easily verified. Alternatively, we could have proven the second inequality stated in the lemma using the fact that $\dfrac{L_{m-2}+L_{m-4}}{5F_{m-1}}=\dfrac{F_{m-3}}{F_{m-1}}$. 
\end{proof} 
\begin{lemma} \label{Lem17}
Let $k$ and $m$ be integers such that $k\geq 5$ and $2\leq m\leq k-1$. Define a function $\Theta_{k,m}\colon\mathbb R\rightarrow\mathbb R$ by \[\Theta_{k,m}(x)=\frac{J_{k,m}(b^k+x)}{(b^k+x)^{\log_b\phi}}.\] For any $x\geq 0$,
\[\frac{d}{dx}\Theta_{k,m}(b^k+x)\geq 0.\] 
\end{lemma} 
\begin{proof} 
We begin with the easily-proven fact that $\dfrac{F_{m-2}}{F_{m-1}}\leq\dfrac23$ for all integers $m\geq 4$. From this, we have \[\frac{H_k(m)}{(k-m)F_{m-1}}=\frac{(k-m)F_m+F_{m+2}}{(k-m)F_{m-1}}\leq\frac{(k-m)F_m+(k-m)F_{m+2}}{(k-m)F_{m-1}}\] \[=\frac{F_m+F_{m+2}}{F_{m-1}}=\frac{4F_{m-1}+3F_{m-2}}{F_{m-1}}=4+3\frac{F_{m-2}}{F_{m-1}}\leq 6.\] One may easily show that \[b^2+b^{-1}\log_b\phi\geq 6\log_b\phi\] for any choice of an integer $b\geq 2$ (if $b\geq 3$, this follows from the observation that $b^2>6\log_b\phi$). In addition, using Definition \ref{Def2} and Lemma \ref{Lem16}, we have $h_m\geq b^{m-2}$ and \[h_{m+1}-h_m\leq\frac{b^m+b}{b+1}\leq\frac{b^m+b^{m-1}}{b+1}=b^{m-1}.\]  
Therefore, \[b^k+h_m\log_b\phi\geq b^k+b^{m-2}\log_b\phi=b^{m-1}(b^{k-m+1}+b^{-1}\log_b\phi)\] \[\geq b^{m-1}(b^2+b^{-1}\log_b\phi)\geq 6b^{m-1}\log_b\phi\geq6(h_{m+1}-h_m)\log_b\phi\] \[\geq\frac{H_k(m)}{(k-m)F_{m-1}}(h_{m+1}-h_m)\log_b\phi.\] Let $T_k(m)=\dfrac{(k-m)F_{m-1}}{h_{m+1}-h_m}$ so that we obtain $b^k+h_m\log_b\phi\geq \dfrac{H_k(m)}{T_k(m)}\log_b\phi$. Choose some $x\geq 0$. We have \[xT_k(m)(1-\log_b\phi)\geq 0\geq T_k(m)\left[\frac{H_k(m)}{T_k(m)}\log_b\phi-h_m\log_b\phi-b^k\right]\] \[=\left[H_k(m)-T_k(m)h_m\right]\log_b\phi-b^kT_k(m).\] Adding $xT_k(m)\log_b\phi+b^kT_k(m)$ to each side of this last inequality yields \[(b^k+x)T_k(m)\geq\log_b\phi\left(xT_k(m)+\left[H_k(m)-T_k(m)h_m\right]\right),\] which we may rewrite as 
\begin{equation} \label{Eq20}
(b^k+x)T_k(m)-\log_b\phi\left[H_k(m)+T_k(m)(x-h_m)\right]\geq 0.
\end{equation} 
\par 
Now, referring to the paragraph immediately following Definition \ref{Def8}, we see that $J_{k,m}(b^k+x)=T_k(m)(x-h_m)+H_k(m)$. Therefore, \eqref{Eq20} becomes $(b^k+x)T_k(m)-J_{k,m}(b^k+x)\log_b\phi\geq 0$. In addition, $\dfrac{d}{dx}J_{k,m}(b^k+x)=T_k(m)$. Consequently, \[\frac{d}{dx}\frac{J_{k,m}(b^k+x)}{(b^k+x)^{\log_b\phi}}=\frac{(b^k+x)^{\log_b\phi}T_k(m)-J_{k,m}(b^k+x)\log_b\phi(b^k+x)^{\log_b\phi-1}}{(b^k+x)^{2\log_b\phi}}\] \[=\frac{(b^k+x)T_k(m)-J_{k,m}(b^k+x)\log_b\phi}{(b^k+x)^{\log_b\phi+1}}\geq 0.\] 
\end{proof} 
We are finally ready to prove the main results of this section. 
\begin{theorem} \label{Thm4} 
If $k,m$, and $n$ are integers such that $2\leq m<k$ and $G_k(m)<n<G_k(m+1)$, then \[s_b(n)<J_{k,m}(n).\]
\end{theorem} 
\begin{proof}
By Definition \ref{Def2} and Definition \ref{Def4}, \[n\leq G_k(m+1)-1=b^k+\sum_{i=0}^{\left\lfloor\frac{m-2}{2}\right\rfloor}b^{m-1-2i}.\] Therefore, we may let $y$ be the smallest element of the set $\left\{0,1,\ldots,\left\lfloor\frac{m-2}{2}\right\rfloor\right\}$ such that $n\leq b^k+\displaystyle{\sum_{i=0}^yb^{m-1-2i}}$. Referring to Definition \ref{Def3}, we find that $s_b(n)\leq\mu(k,m-1,y)$ because $n\in I(k,m-1,y)$. Hence, it suffices to prove that 
\begin{equation} \label{Eq10} 
\mu(k,m-1,y)<J_{k,m}(n). 
\end{equation} It follows from Theorem \ref{Thm3}, Definition \ref{Def7}, Definition \ref{Def4}, and the identity $F_u=\dfrac15(L_{u+1}+L_{u-1})$ that \[\mu(k,m-1,0)=V(k,m-1,0)=\frac15((k-m+1)(2L_{m+1}-L_m)+2L_m+L_{m+1})\] \[=(k-m)\left(\frac15(2L_{m+1}-L_m)\right)+\frac15(2L_{m+2}+L_{m-1})\] \[=(k-m)\left(\frac15(L_{m+1}+L_{m-1})\right)+\frac15(L_{m+3}+L_{m+1})\] \[=(k-m)F_m+F_{m+2}=H_k(m).\] This shows that if $y=0$, then \[s_b(n)\leq\mu(k,m-1,0)=H_k(m)=J_{k,m}(G_k(m))<J_{k,m}(n),\] where we have used the trivial fact that $J_{k,m}$ is an increasing function. Therefore, we may assume $y>0$. Because of the way we chose $y$, this means that $n>b^k+\displaystyle{\sum_{i=0}^{y-1}b^{m-1-2i}}\geq b^k+b^{m-1}$. Note that this also forces $m\geq 4$ because $y\leq\left\lfloor\frac{m-2}{2}\right\rfloor$. By a similar token, it follows from Lemma \ref{Lem19} that \[s_b(n)=s_2(n)\leq H_k(m)=J_{k,m}(G_k(m))<J_{k,m}(n)\] if $b=2$ and $n<2^k+G_{m-1}(m-3)$. Therefore, we may assume $n\geq 2^k+G_{m-1}(m-3)$ if $b=2$. If $b=2$, $m\in\{5,7\}$, and $n=2^k+G_{m-1}(m-3)$, then the desired result is simply Lemma \ref{Lem20}. This means that if $b=2$ and $m\in\{5,7\}$, then we may assume $n\geq 2^k+G_{m-1}(m-3)+1$. We have two cases to consider.
\par 
\noindent 
Case 1: In this case, assume that $\mu(k,m-1,y)=V(k,m-1,y)$. Recall Definition \ref{Def7} to see that \[V(k,m-1,y)=\frac 15((k-m+1)(2L_{m+1}-L_{m-4y})+2L_m+L_{m-4y+1})\] \[=\frac 15((k-m)(2L_{m+1}-L_{m-4y})+2L_{m+2}+L_{m-4y-1}).\] We now use the fact that $L_{m-4y-1}<L_{m-1}$ (which follows immediately from Lemma \ref{Lem13}) as well as the identity $\dfrac15(L_{u+1}+L_{u-1})=F_u$ to see that \[V(k,m-1,y)<\frac15((k-m)(2L_{m+1}-L_{m-4y})+2L_{m+2}+L_{m-1})\] \[=\frac15(k-m)(2L_{m+1}-L_{m-4y})+\frac15(L_{m+3}+L_{m+1})\] \[=\frac15(k-m)(2L_{m+1}-L_{m-4y})+F_{m+2}.\] Because \[J_{k,m}(n)=\frac{F_{m-1}(k-m)}{h_{m+1}-h_m}\left(n-G_k(m)\right)+H_k(m)\] \[=\frac{F_{m-1}(k-m)}{h_{m+1}-h_m}\left(n-b^k-h_m\right)+F_m(k-m)+F_{m+2},\] we simply need to show that \[\frac 15(k-m)(2L_{m+1}-L_{m-4y})\leq\frac{F_{m-1}(k-m)}{h_{m+1}-h_m}\left(n-b^k-h_m\right)+F_m(k-m)\] in order to obtain \eqref{Eq10}. After dividing each side by $k-m$ and using the identity \[\frac 15(2L_{m+1}-L_{m-4y})=F_{m+1}-\dfrac15(L_{m-2}+L_{m-4y}),\] we find that this last inequality becomes \[F_{m+1}-\frac15(L_{m-2}+L_{m-4y})\leq F_{m-1}\left(1-\frac{h_{m+1}-(n-b^k)}{h_{m+1}-h_m}\right)+F_m,\] which we may rewrite as 
\begin{equation} \label{Eq12} 
\frac{h_{m+1}-(n-b^k)}{h_{m+1}-h_m}\leq\frac{L_{m-2}+L_{m-4y}}{5F_{m-1}}
\end{equation} 
after subtracting $F_{m+1}$ from each side and then rearranging terms. We will prove \eqref{Eq12} in each of the following two subcases. 
\par 
\noindent 
Subcase 1: In this subcase, assume $b=2$ and $y=1$. Recall that we mentioned at the beginning of the proof that we may assume $n\geq 2^k+G_{m-1}(m-3)$ when $b=2$. Furthermore, we may assume $n\geq 2^k+G_{m-1}(m-3)+1$ if $m\in\{5,7\}$. If $m=5$, then \[\frac{h_{m+1}-(n-b^k)}{h_{m+1}-h_m}\leq\frac{h_{m+1}-(G_{m-1}(m-3)+1)}{h_{m+1}-h_m}=\frac{h_6-(G_4(2)+1)}{h_6-h_5}\] \[=\frac{(2^4+2^2+1)-(2^4+1+1)}{(2^4+2^2+1)-(2^3+2+1)}=\frac{3}{10}<\frac13=\frac{L_3+L_1}{5F_4}=\frac{L_{m-2}+L_{m-4y}}{5F_{m-1}}.\] If $m=7$, then 
\[\frac{h_{m+1}-(n-b^k)}{h_{m+1}-h_m}\leq\frac{h_{m+1}-(G_{m-1}(m-3)+1)}{h_{m+1}-h_m}=\frac{h_8-(G_6(4)+1)}{h_8-h_7}\] \[=\frac{(2^6+2^4+2^2+1)-(2^6+2^2+1+1)}{(2^6+2^4+2^2+1)-(2^5+2^3+2+1)}=\frac{5}{14}<\frac38=\frac{L_5+L_3}{5F_6}\] \[=\frac{L_{m-2}+L_{m-4y}}{5F_{m-1}}.\] This proves \eqref{Eq12} if $m\in\{5,7\}$, so we will assume $m\not\in\{5,7\}$. By Definition \ref{Def2}, \[h_{m+1}=1+\sum_{i=0}^{\left\lfloor\frac{m-2}{2}\right\rfloor}2^{m-1-2i}=1+2^{m-1}+2^{m-3}+\sum_{i=2}^{\left\lfloor\frac{m-2}{2}\right\rfloor}2^{m-1-2i}\] \[=1+2^{m-1}+2^{m-3}+\sum_{i=0}^{\left\lfloor\frac{m-6}{2}\right\rfloor}2^{m-5-2i}=2^{m-3}+2^{m-1}+h_{m-3}\] \[=2^{m-3}+G_{m-1}(m-3).\] Therefore, using Lemma \ref{Lem16} to write $h_{m+1}-h_m=\dfrac{2^m+(-1)^m\cdot2}{3}$, we have
\[\frac{h_{m+1}-(n-b^k)}{h_{m+1}-h_m}\leq\frac{h_{m+1}-G_{m-1}(m-3)}{h_{m+1}-h_m}=\frac{2^{m-3}}{h_{m+1}-h_m}=3\frac{2^{m-3}}{2^m+(-1)^m\cdot2}.\] If $m$ is even, then \[\frac{h_{m+1}-(n-b^k)}{h_{m+1}-h_m}\leq 3\frac{2^{m-3}}{2^m+2}<3\frac{2^{m-3}}{2^m}=\frac38<\frac{8}{21}.\] If $m$ is odd, then $m\geq 9$, so \[\frac{h_{m+1}-(n-b^k)}{h_{m+1}-h_m}\leq 3\frac{2^{m-3}}{2^m-2}=\frac{3}{8-2^{4-m}}\leq\frac{3}{8-2^{4-9}}=\frac{32}{85}<\frac{8}{21}.\] Either way, $\dfrac{h_{m+1}-(n-b^k)}{h_{m+1}-h_m}<\dfrac{8}{21}$. Lemma \ref{Lem14} tells us that \\ $\displaystyle{\frac{8}{21}\leq\frac{L_{m-2}+L_{m-4}}{5F_{m-1}}}$, so we obtain \eqref{Eq12}. 
\par 
\noindent
Subcase 2: In this subcase, suppose either $y>1$ or $b\neq 2$. It follows from our choice of $y$ that $n>b^k+\displaystyle{\sum_{i=0}^{y-1}b^{m-1-2i}}$. Using Definition \ref{Def2} and Lemma \ref{Lem16}, we see that \[h_{m+1}-(n-b^k)=1+\sum_{i=0}^{\lfloor\frac{m-2}{2}\rfloor}b^{m-1-2i}-(n-b^k)<1+\sum_{i=0}^\infty b^{m-1-2i}-(n-b^k)\] \[<1+\sum_{i=0}^\infty b^{m-1-2i}-\sum_{i=0}^{y-1}b^{m-1-2i}=1+\sum_{i=y}^\infty b^{m-1-2i}=1+\frac{b^{m+1-2y}}{b^2-1}\] and \[h_{m+1}-h_m\geq\frac{b^m-b}{b+1}.\]  Therefore, 
\begin{equation} \label{Eq13} 
\frac{h_{m+1}-(n-b^k)}{h_{m+1}-h_m}<\frac{1+b^{m+1-2y}/(b^2-1)}{(b^m-b)/(b+1)}=\frac{b^2+b^{m+1-2y}-1}{(b^m-b)(b-1)}.
\end{equation} If we treat $m$ as a continuous real variable, then \[\frac{\partial}{\partial m}\left(\frac{b^2+b^{m+1-2y}-1}{(b^m-b)(b-1)}\right)=\frac{1}{b-1}\frac{\partial}{\partial m}\left(\frac{b^2+b^{m+1-2y}-1}{b^m-b}\right)\] \[=\frac{1}{b-1}\frac{(b^m-b)b^{m+1-2y}\log b-(b^2+b^{m+1-2y}-1)b^m\log b}{(b^m-b)^2}\] \[=\frac{\log b}{(b-1)(b^m-b)^2}(-b^{m+2-2y}-b^{m+2}+b^m)<0.\] This shows that $\dfrac{b^2+b^{m+1-2y}-1}{(b^m-b)(b-1)}$ is decreasing in $m$. 
\par
Suppose $y\geq 2$. This forces $m\geq 6$ because $y\leq\left\lfloor\frac{m-2}{2}\right\rfloor$. Because $\dfrac{b^2+b^{m+1-2y}-1}{(b^m-b)(b-1)}$ is decreasing in $m$, we see from \eqref{Eq13} that \[\frac{h_{m+1}-(n-b^k)}{h_{m+1}-h_m}<\dfrac{b^2+b^{m+1-2y}-1}{(b^m-b)(b-1)}\leq\frac{b^2+b^{7-2y}-1}{(b^6-b)(b-1)}\leq\frac{b^2+b^3-1}{(b^6-b)(b-1)}.\] If $b=2$, then \[\frac{h_{m+1}-(n-b^k)}{h_{m+1}-h_m}<\frac{b^2+b^3-1}{(b^6-b)(b-1)}=\frac{11}{62}<\sqrt5-2.\] If $b\geq 3$, then \[\frac{h_{m+1}-(n-b^k)}{h_{m+1}-h_m}<\frac{b^2+b^3-1}{(b^6-b)(b-1)}<\frac{b^3+b^2+b+1+b^{-1}}{(b^6-b)(b-1)}\] \[=\frac{b^4+b^3+b^2+b+1}{b^2(b^5-1)(b-1)}=\frac{1}{b^2(b-1)^2}\leq\frac{1}{3^2(3-1)^2}=\frac{1}{36}<\sqrt5-2.\] No matter the value of $b$, we find from Lemma \ref{Lem14} that \[\frac{h_{m+1}-(n-b^k)}{h_{m+1}-h_m}<\sqrt5-2<\frac{L_{m-2}+L_{m-4y}}{5F_{m-1}},\] which is \eqref{Eq12}.
\par 
We have proven \eqref{Eq12} when $y\geq 2$, so assume $y=1$. We assumed either $y>1$ or $b\neq 2$, so we must have $b\geq 3$. Since $\dfrac{b^2+b^{m+1-2y}-1}{(b^m-b)(b-1)}$ is decreasing in $m$ and $m\geq 4$, it follows from \eqref{Eq13} that \[\frac{h_{m+1}-(n-b^k)}{h_{m+1}-h_m}<\dfrac{b^2+b^{m+1-2y}-1}{(b^m-b)(b-1)}\leq\frac{b^2+b^{5-2y}-1}{(b^4-b)(b-1)}=\frac{b^2+b^3-1}{(b^4-b)(b-1)}.\] If $b=3$, then \[\frac{h_{m+1}-(n-b^k)}{h_{m+1}-h_m}<\frac{b^2+b^3-1}{(b^4-b)(b-1)}=\frac{35}{156}<\sqrt5-2,\] which proves \eqref{Eq12} with the help of Lemma \ref{Lem14}. If $b\geq 4$, then \[\frac{h_{m+1}-(n-b^k)}{h_{m+1}-h_m}<\frac{b^2+b^3-1}{(b^4-b)(b-1)}<\frac{b^3+b^2+b}{(b^4-b)(b-1)}=\frac{b^2+b+1}{(b^3-1)(b-1)}\] \[=\frac{1}{(b-1)^2}\leq\frac{1}{(4-1)^2}=\frac 19<\sqrt5-2,\] which proves \eqref{Eq12} once again.
\par 
\noindent 
Case 2: Here, assume $\mu(k,m-1,y)\neq V(k,m-1,y)$. Referring to Theorem \ref{Thm3}, we see that $\nu(k,m-1,y)=\gamma(k,m-1,x)+1$ and $\mu(k,m-1,y)=V(k,m-1,x)$ for some $x<y$. One finds from Definition \ref{Def6} that $\gamma(k,m-1,x)+1>b^k+\displaystyle{\sum_{i=0}^{x-1}b^{m-1-2i}}$. By Lemma \ref{Lem11}, $\gamma(k,m-1,x)+1\in I(k,m-1,x)$. In other words, if we let $y'$ be the smallest element of the set $\left\{0,1,\ldots,\left\lfloor\frac{m-2}{2}\right\rfloor\right\}$ such that $\gamma(k,m-1,x)+1\leq b^k+\displaystyle{\sum_{i=0}^{y'}b^{m-1-2i}}$, then $y'=x$. By definition, $\mu(k,m-1,x)$ is the maximum value of $s_b(j)$ as $j$ ranges over all elements of $I(k,m-1,x)$. Because $\gamma(k,m-1,x)+1\in I(k,m-1,x)$, this means that \[\mu(k,m-1,x)\geq s_b(\gamma(k,m-1,x)+1)=s_b(\nu(k,m-1,y))=\mu(k,m-1,y),\] where we have used Definition \ref{Def3} to deduce the last equality. Similarly, $\mu(k,m-1,y)$ is the maximum value of $s_b(j)$ as $j$ ranges over all elements of $I(k,m-1,y)$. Since $x<y$, $I(k,m-1,x)\subseteq I(k,m-1,y)$. Therefore, $\mu(k,m-1,x)\leq\mu(k,m-1,y)$. This shows that \[\mu(k,m-1,x)=\mu(k,m-1,y)=V(k,m-1,x).\]
\par 
Now, choose some integer $n'\in \left(b^k+\displaystyle{\sum_{i=0}^{x-1}b^{m-1-2i}},b^k+\displaystyle{\sum_{i=0}^xb^{m-1-2i}}\right]$ (note that $\gamma(k,m-1,x)+1$ is in this interval). Because $x$ is the smallest element of the set $\left\{0,1,\ldots,\left\lfloor\frac{m-2}{2}\right\rfloor\right\}$ satisfying $n'\leq b^k+\displaystyle{\sum_{i=0}^xb^{m-1-2i}}$ and $\mu(k,m-1,x)=V(k,m-1,x)$, it follows from Case 1 that $s_b(n')<J_{k,m}(n')$. We may set $n'=\gamma(k,m-1,x)+1$, so $s_b(\gamma(k,m-1,x)+1)<J_{k,m}(\gamma(k,m-1,x)+1)$. Using the fact that $\gamma(k,m-1,x)+1\leq b^k+\displaystyle{\sum_{i=0}^xb^{m-1-2i}}\leq b^k+\displaystyle{\sum_{i=0}^{y-1}b^{m-1-2i}}<n$, we have \[\mu(k,m-1,y)=s_b(\gamma(k,m-1,x)+1)<J_{k,m}(\gamma(k,m-1,x)+1)<J_{k,m}(n).\] This proves \eqref{Eq10} and completes the proof of the theorem. 
\end{proof} 
\begin{corollary} 
We have \[\limsup_{n\rightarrow\infty}\frac{s_b(n)}{n^{\log_b\phi}}=\frac{(b^2-1)^{\log_b\phi}}{\sqrt 5}.\]
\end{corollary}
\begin{proof} 
Corollary \ref{Cor1} states that \[\limsup_{n\rightarrow\infty}\frac{s_b(n)}{n^{\log_b\phi}}\geq\frac{(b^2-1)^{\log_b\phi}}{\sqrt 5},\] so we will now prove the reverse inequality. For each integer $k\geq 3$, let $u_k=G_k(k)=b^k+h_k$. Let $\theta(x)=\lceil\log_b(x)\rceil-1$ for each $x>0$. Recall that we showed in the proof of Corollary \ref{Cor1} that 
\begin{equation} \label{Eq19} 
\lim_{k\rightarrow\infty}\frac{s_b(u_k)}{u_k^{\log_b\phi}}=\frac{(b^2-1)^{\log_b\phi}}{\sqrt5}. 
\end{equation} We will show that 
\begin{equation} \label{Eq18} 
\frac{s_b(n)}{n^{\log_b\phi}}\leq\frac{s_b(u_{\theta(n)})}{(u_{\theta(n)})^{\log_b\phi}}\hspace{.2cm}\text{ for all }n>b^5.
\end{equation} It will then follow from \eqref{Eq19} and \eqref{Eq18} that \[\limsup_{n\rightarrow\infty}\frac{s_b(n)}{n^{\log_b\phi}}\leq\limsup_{n\rightarrow\infty}\frac{s_b(u_{\theta(n)})}{(u_{\theta(n)})^{\log_b\phi}}=\limsup_{k\rightarrow\infty}\frac{s_b(u_k)}{u_k^{\log_b\phi}}=\frac{(b^2-1)^{\log_b\phi}}{\sqrt5},\] which will complete the proof. In order to derive \eqref{Eq18}, let us choose some integer $n>b^5$. Let $\theta=\theta(n)$, and note that $b^{\theta}<n\leq b^{\theta+1}$. It follows from Proposition \ref{Prop1} that \[s_b(n)\leq F_{\theta+2}=s_b(b^{\theta}+h_\theta)=s_b(u_\theta)\] (if $n=b^{\theta+1}$, then Proposition \ref{Prop1} does not apply, but the inequality still holds because $s_b(n)=1<F_{\theta+2}$). If $n\geq u_\theta$, then this shows that \[\frac{s_b(n)}{n^{\log_b\phi}}\leq\frac{s_b(u_\theta)}{n^{\log_b\phi}}\leq\frac{s_b(u_\theta)}{u_\theta^{\log_b\phi}},\] which is the inequality we seek to prove. Therefore, we will assume $n<u_\theta$. 
\par 
For any $t\in\{2,3,\ldots,\theta-1\}$, we may use Lemma \ref{Lem17} to see that \[\Theta_{\theta,t}(G_\theta(t))=\Theta_{\theta,t}(b^\theta+h_t)\leq\Theta_{\theta,t}(b^\theta+h_{t+1})=\Theta_{\theta,t}(G_\theta(t+1)),\] where we have preserved the notation from that lemma. Furthermore, with the help of Definition \ref{Def8}, we find that \[\Theta_{\theta,t}(G_\theta(t+1))=\frac{J_{\theta,t}(G_\theta(t+1))}{G_\theta(t+1)^{\log_b\phi}}=\frac{H_\theta(t+1)}{G_\theta(t+1)^{\log_b\phi}}\] 
\begin{equation} \label{Eq21} 
=\frac{J_{\theta,t+1}(G_\theta(t+1))}{G_\theta(t+1)^{\log_b\phi}}=\Theta_{\theta,t+1}(G_\theta(t+1))
\end{equation} for all $t\in\{2,3,\ldots,\theta-1\}$. Therefore, \[\Theta_{\theta,t}(G_\theta(t))\leq\Theta_{\theta,t+1}(G_\theta(t+1))\] for all $t\in\{2,3,\ldots,\theta-1\}$. That is, 
\begin{equation} \label{Eq22}
\Theta_{\theta,2}(G_\theta(2))\leq\Theta_{\theta,3}(G_\theta(3))\leq\cdots\leq\Theta_{\theta,\theta}(G_\theta(\theta)).
\end{equation} Because $b^\theta<n<u_\theta=G_\theta(\theta)$, there exists some $m\in\{2,3,\ldots,\theta-1\}$ such that $G_\theta(m)\leq n<G_\theta(m+1)$. In other words, we may write $n=b^\theta+x$ for some $x\in\{h_m,h_m+1,\ldots,h_{m+1}-1\}$. By Lemma \ref{Lem17} and \eqref{Eq21}, \[\Theta_{\theta,m}(n)=\Theta_{\theta,m}(b^\theta+x)\leq\Theta_{\theta,m}(b^\theta+h_{m+1})\] \[=\Theta_{\theta,m}(G_\theta(m+1))=\Theta_{\theta,m+1}(G_\theta(m+1)).\] It follows from \eqref{Eq22} that $\Theta_{\theta,m+1}(G_\theta(m+1))\leq\Theta_{\theta,\theta}(G_\theta(\theta))$, so 
\begin{equation} \label{Eq23} 
\Theta_{\theta,m}(n)\leq\Theta_{\theta,\theta}(G_\theta(\theta)).
\end{equation} If $n=G_\theta(m)$, then it follows easily from Theorem \ref{Thm1} and Definition \ref{Def8}, that \[s_b(n)=s_b(G_\theta(m))=H_\theta(m)=J_{\theta,t}(G_\theta(m))=J_{\theta,m}(n).\] If $n\neq G_\theta(m)$, then $G_\theta(m)<n<G_{\theta}(m+1)$, so Theorem \ref{Thm4} shows that $s_b(n)<J_{\theta,m}(n)$. Either way, $s_b(n)\leq J_{\theta,m}(n)$. Hence, using \eqref{Eq23}, we obtain \[\frac{s_b(n)}{n^{\log_b\phi}}\leq\frac{J_{\theta,m}(n)}{n^{\log_b\phi}}=\Theta_{\theta,m}(n)\leq\Theta_{\theta,\theta}(G_\theta(\theta))\] \[=\frac{J_{\theta,\theta}(G_\theta(\theta))}{G_\theta(\theta)^{\log_b\phi}}=\frac{s_b(G_\theta(\theta))}{G_\theta(\theta)^{\log_b\phi}}=\frac{s_b(u_\theta)}{u_\theta^{\log_b\phi}},\] which proves \eqref{Eq18}. 
\end{proof}  
\begin{theorem} \label{Thm5}
For any real $x\geq 0$ and any integer $k\geq 3$, let \[\beta(x)=\log_b((b^2-1)x+1)\] and \[f_k(x)=\frac{1}{\sqrt5}\left[(k-\beta(x))(\phi^{\beta(x)}+\phi^{-\beta(x)})+\phi^{\beta(x)+2}+\phi^{-\beta(x)-2}\right].\] If $k$ and $n$ are positive integers such that $k\geq 3$ and $b^k<n\leq b^k+h_k$, then \[s_b(n)<f_k(n-b^k).\]
\end{theorem} 
\begin{proof} 
Let $k\geq 3$ be an integer. Let $\alpha=\dfrac{b^k+b^2-b-1}{b^2-1}$, and note that $h_k\leq\alpha$. We will show that $f_k$ is increasing and concave down on the open interval $(1,\alpha)$. The desired inequality will then follow quite easily. Observe that we may write $f_k(x)=g_k(\beta(x))$, where \[g_k(x)=\frac{1}{\sqrt5}\left[(k-x)(\phi^x+\phi^{-x})+\phi^{x+2}+\phi^{-x-2}\right].\] If $x\leq\log_b(b^k+b^2-b)$, then \[k-x\geq k-\log_b(b^k+b^2-b)=-\log_b\left(1+\frac{b-1}{b^{k-1}}\right)\geq-\frac{b-1}{b^{k-1}\log b}\] \[\geq-\frac{b-1}{b^2\log b}\geq-\frac{1}{4\log2},\] where we have used the inequality $\log_b(1+u)\leq\dfrac{u}{\log b}$ that holds for all $u>-1$ (as well as the inequalities $k\geq 3$ and $b\geq 2$). Therefore, if $2\leq x\leq\log_b(b^k+b^2-b)$, then \[g_k'(x)=\frac{1}{\sqrt5}\left[(k-x)(\phi^x-\phi^{-x})\log\phi-(\phi^x+\phi^{-x})+(\phi^{x+2}-\phi^{-x-2})\log\phi\right]\] \[\geq\frac{1}{\sqrt5}\left[-\frac{1}{4\log 2}(\phi^x-\phi^{-x})\log\phi-(\phi^x+\phi^{-x})+(\phi^{x+2}-\phi^{-x-2})\log\phi\right]\] \[=\frac{1}{\sqrt5}\left[\phi^x\left(\phi^2\log\phi-\frac{\log\phi}{4\log2}-1\right)+\phi^{-x}\left(-\phi^{-2}\log\phi+\frac{\log\phi}{4\log2}-1\right)\right]\] \[=\frac{1}{\sqrt5}\left[C_1\phi^x+C_2\phi^{-x}\right],\] where \[C_1=\phi^2\log\phi-\frac{\log\phi}{4\log2}-1\approx 0.086\ldots\] and \[C_2=-\phi^{-2}\log\phi+\frac{\log\phi}{4\log2}-1\approx-1.010\ldots.\] If $x\geq2$, then $C_1\phi^x\geq C_1\phi^2$ and $C_2\phi^{-x}\geq C_2\phi^{-2}$. This means that \[g_k'(x)\geq\frac{1}{\sqrt5}\left[C_1\phi^x+C_2\phi^{-x}\right]\geq C_1\phi^2+C_2\phi^{-2}>0\] whenever $2\leq x\leq\log_b(b^k+b^2-b)$. Now, if $1\leq x\leq \alpha$, then $2\leq\beta(x)\leq\log_b(b^k+b^2-b)$ and $\beta'(x)>0$. Consequently, \[f_k'(x)=\frac{d}{dx}g_k(\beta(x))=g_k'(\beta(x))\beta'(x)>0\] for all $x\in[1,\alpha]$. This shows that $f_k$ is increasing on the open interval $(1,\alpha)$.
\par 
We now wish to show that $f_k$ is concave down on the interval $(1,\alpha)$. We first calculate \[\beta'(x)=\frac{1}{\log b}\frac{b^2-1}{(b^2-1)x+1}\] and \[\beta''(x)=-\frac{1}{\log b}\left(\frac{b^2-1}{(b^2-1)x+1}\right)^2=-(\beta'(x))^2\log  b\] to obtain \[f_k''(x)=\frac{d^2}{d^2x}g_k(\beta(x))=\frac{d}{dx}(g_k'(\beta(x))\beta'(x))\] \[=g_k'(\beta(x))\beta''(x)+(\beta'(x))^2g_k''(\beta(x))\] \[=(\beta'(x))^2(g_k''(\beta(x))-g_k'(\beta(x))\log b).\] Therefore, to show that $f_k$ is concave down on the interval $(1,\alpha)$, we just need to show that $g_k'(\beta(x))\log b>g_k''(\beta(x))$ for all $x\in(1,\alpha)$. To do so, it suffices to show that 
\begin{equation} \label{Eq14}
\frac{\sqrt5\log b}{\log\phi}g_k'(x)>\frac{\sqrt5}{\log\phi}g_k''(x)
\end{equation} for all $x\in(2,\log_b(b^k+b^2-b))$. 
\par 
Suppose $x\in(2,\log_b(b^k+b^2-b))$. We have \[-(\phi^x+\phi^{-x})+(\phi^{x+2}-\phi^{-x-2})\log\phi=\phi^x(\phi^2\log\phi-1)-\phi^{-x}(\phi^{-2}\log\phi+1)\] \[\geq\phi^2(\phi^2\log\phi-1)-\phi^{-x}(\phi^{-2}\log\phi+1)\] \[\geq\phi^2(\phi^2\log\phi-1)-\phi^{-2}(\phi^{-2}\log\phi+1)>0,\] so \[g_k'(x)=\frac{1}{\sqrt5}\left[(k-x)(\phi^x-\phi^{-x})\log\phi-(\phi^x+\phi^{-x})+(\phi^{x+2}-\phi^{-x-2})\log\phi\right]\] \[>\frac{1}{\sqrt5}\left[(k-x)(\phi^x-\phi^{-x})\log\phi\right].\] In other words,
\begin{equation} \label{Eq15}  
\frac{\sqrt5\log b}{\log\phi}g_k'(x)>(k-x)(\phi^x-\phi^{-x})\log b.
\end{equation} Using the inequality $\log_b(1+u)\leq\dfrac{u}{\log b}$, which holds for all $u>-1$, we find that \[k-x>k-\log_b(b^k+b^2-b)=-\log_b\left(1+\frac{b-1}{b^{k-1}}\right)\geq-\frac{b-1}{b^{k-1}\log b},\] so it follows from the assumption that $k\geq 3$ that 
\begin{equation} \label{Eq16} 
k-x>-\frac{b-1}{b^2\log b}.  
\end{equation} 
In addition, since $x>2$, 
\begin{equation} \label{Eq17} 
\phi^x(\log b-\log\phi)-\phi^{-x}(\log b+\log\phi)>\phi^2(\log b-\log\phi)-\phi^{-2}(\log b+\log\phi)>0.
\end{equation} 
Combining \eqref{Eq15}, \eqref{Eq16}, and \eqref{Eq17} with the fact that
\[\frac{\sqrt5}{\log\phi}g_k''(x)=\log\phi\left[(k-x)(\phi^x+\phi^{-x})-2\frac{\phi^x-\phi^{-x}}{\log\phi}+(\phi^{x+2}+\phi^{-x-2})\right]\] yields \[\frac{\sqrt5\log b}{\log\phi}g_k'(x)-\frac{\sqrt5}{\log\phi}g_k''(x)>(k-x)(\phi^x-\phi^{-x})\log b-(k-x)(\phi^x+\phi^{-x})\log\phi\] \[+2(\phi^x-\phi^{-x})-(\phi^{x+2}+\phi^{-x-2})\log\phi\] \[=(k-x)(\phi^x(\log b-\log\phi)-\phi^{-x}(\log b+\log\phi))\] \[+\phi^x(2-\phi^2\log\phi)-\phi^{-x}(2+\phi^{-2}\log\phi)\] \[>-\frac{b-1}{b^2\log b}(\phi^x(\log b-\log\phi)-\phi^{-x}(\log b+\log\phi))\] \[+\phi^x(2-\phi^2\log\phi)-\phi^{-x}(2+\phi^{-2}\log\phi)\] \[=A_1\phi^x-A_2\phi^{-x},\] where \[A_1=2-\phi^2\log\phi-\frac{b-1}{b^2\log b}(\log b-\log\phi)\] and \[A_2=2+\phi^{-2}\log\phi-\frac{b-1}{b^2\log b}(\log b+\log\phi).\] Because $b\geq 2$, it is easy to see that \[A_1>2-\phi^2\log\phi-\frac{b-1}{b^2}\geq 2-\phi^2\log\phi-\frac14>0.49\] and \[A_2<2+\phi^{-2}\log\phi<2.19.\] Since $x>2$, \[\frac{\sqrt5\log b}{\log\phi}g_k'(x)-\frac{\sqrt5}{\log\phi}g_k''(x)>A_1\phi^x-A_2\phi^{-x}>0.49\phi^x-2.19\phi^{-x}\] \[>0.49\phi^2-2.19\phi^{-2}>0,\] which proves \eqref{Eq14}. 
\par 
We have shown that $f_k$ is increasing and concave down on the interval $(1,\alpha)$. Now, suppose $m$ is odd and $3\leq m\leq k$. Because $m$ is odd, one may easily show that $h_m=1+\displaystyle{\sum_{i=0}^{\left\lfloor\frac{m-3}{2}\right\rfloor}b^{m-2-2i}}=\dfrac{b^m+b^2-b-1}{b^2-1}$. Since \[G_k(m)-b^k=h_m=\frac{b^m+b^2-b-1}{b^2-1}>\frac{b^m-1}{b^2-1},\] we may use the fact that $f_k$ is increasing on $(1,\alpha)$ to see that \[f_k(G_k(m)-b^k)>f_k\left(\frac{b^m-1}{b^2-1}\right)=g_k\left(\beta\left(\frac{b^m-1}{b^2-1}\right)\right)=g_k(m)\] \[=\frac{1}{\sqrt5}\left[(k-m)(\phi^m+\phi^{-m})+\phi^{m+2}+\phi^{-m-2}\right]\] \[=(k-m)\frac{\phi^m-(-1/\phi)^m}{\sqrt5}+\frac{\phi^{m+2}-(-1/\phi)^{m+2}}{\sqrt5}\] \[=(k-m)F_m+F_{m+2}=H_k(m).\] Similarly, if $m$ is even and $2\leq m\leq k$, then \[f_k(G_k(m)-b^k)=f_k(h_m)=f_k\left(\frac{b^m-1}{b^2-1}\right)=g_k\left(\beta\left(\frac{b^m-1}{b^2-1}\right)\right)=g_k(m)\] \[=\frac{1}{\sqrt5}\left[(k-m)(\phi^m+\phi^{-m})+\phi^{m+2}+\phi^{-m-2}\right]\] \[>\frac{1}{\sqrt5}\left[(k-m)(\phi^m-\phi^{-m})+\phi^{m+2}-\phi^{-m-2}\right]\] \[=(k-m)\frac{\phi^m-(-1/\phi)^m}{\sqrt5}+\frac{\phi^{m+2}-(-1/\phi)^{m+2}}{\sqrt5}\] \[=(k-m)F_m+F_{m+2}=H_k(m).\] Hence, $f_k(G_k(m)-b^k)>H_k(m)$ for all $m\in\{2,3,\ldots,k\}$. We may now prove that $s_b(n)<f_k(n-b^k)$ for all $n\in\{b^k+1,b^k+2,\ldots,b^k+h_k\}$. Choose such an integer $n$. If $n=G_k(m)$ for some $m\in\{2,3,\ldots,k\}$, then it follows from Theorem \ref{Thm1} and the preceding discussion that \[s_b(n)=H_k(m)<f_k(n-b^k).\] Therefore, we will assume $n\neq G_k(m)$ for all $m\in\{2,3,\ldots,k\}$. Note that there exists some $m\in\{2,3,\ldots,k-1\}$ such that $G_k(m)<n<G_k(m+1)$. Let $\mathcal C$ be the curve $\{(x,f_k(x-b^k))\colon G_k(m)\leq x\leq G_k(m+1)\}$, and let $\mathcal L=\{J_{k,m}(x)\colon G_k(m)\leq x\leq G_k(m+1)\}$ be the line segment connecting the points $(G_k(m),H_k(m))$ and $(G_k(m+1),H_k(m+1))$. Because $f_k(x)$ is concave down on the interval $(1,\alpha)$, the curve $\mathcal C$ is concave down. Since $f_k(G_k(m)-b^k)>H_k(m)$ and $f_k(G_k(m+1)-b^k)>H_k(m+1)$, the curve $\mathcal C$ must lie above the line segment $\mathcal L$. In particular, $J_{k,m}(n)<f_k(n-b^k)$. By Theorem \ref{Thm4}, \[s_b(n)<J_{k,m}(n)<f_k(n-b^k),\] as desired. 
\end{proof} 
\begin{figure} 
\includegraphics[height=70mm]{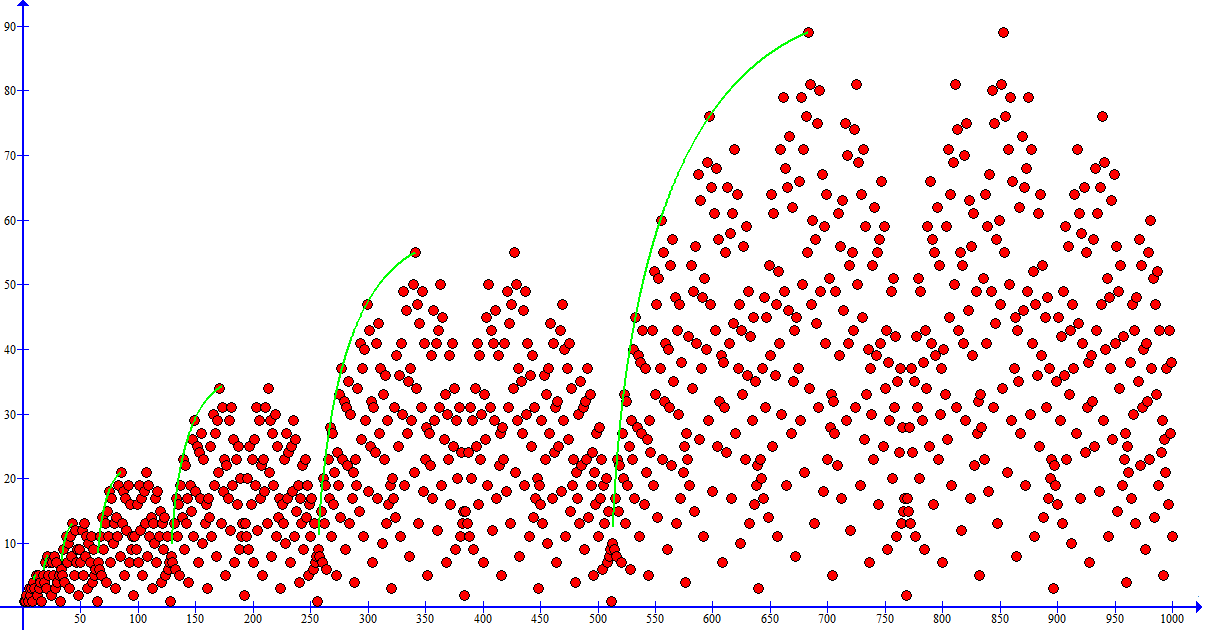} \\
\\ \\
\includegraphics[height=70mm]{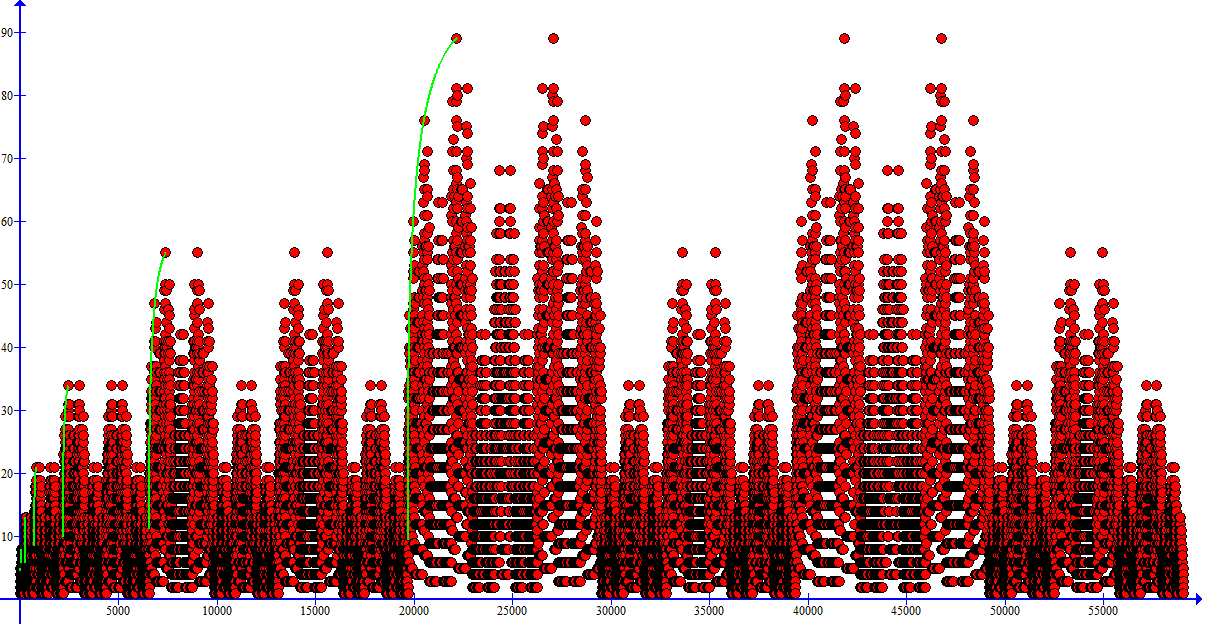}
\caption{Two plots of $s_b(n)$ for $1\leq n\leq b^{10}$. The top plot uses the value $b=2$, while the bottom uses $b=3$. For each $k\in\{3,4,\ldots,9\}$, the graph of $f_k(x-b^k)$ for $b^k\leq x\leq b^k+h_k$ is shown in green.} \label{Fig2}
\end{figure} 
\begin{figure} 
\includegraphics[height=70mm]{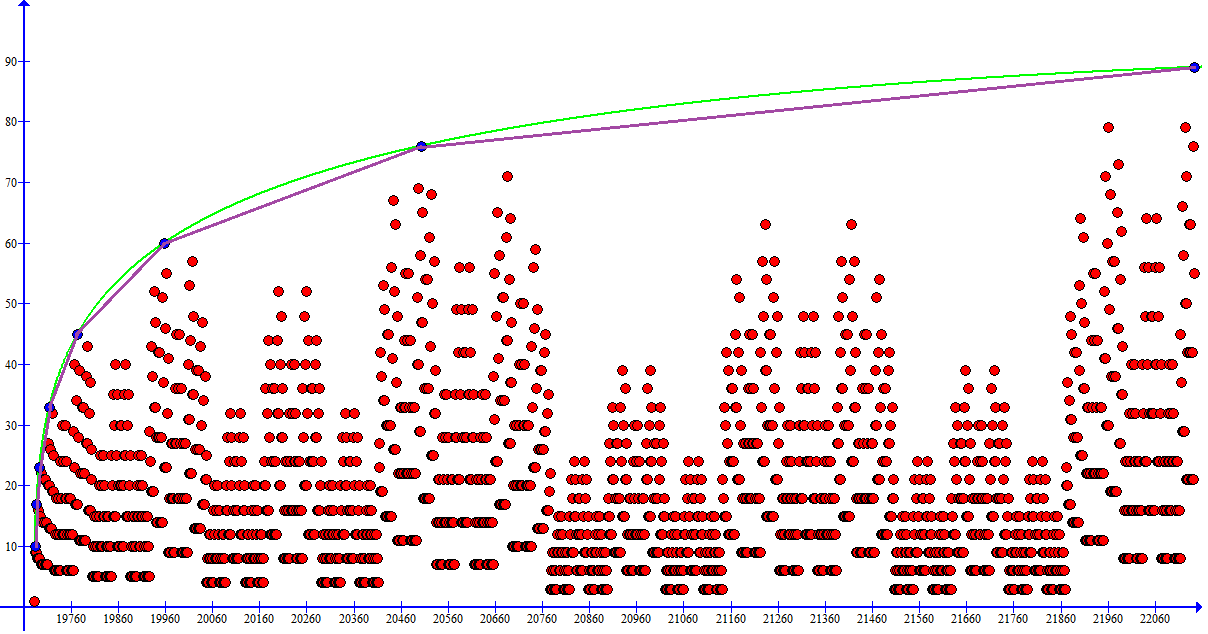}
\caption{A plot of $s_3(n)$ for $3^9\leq n\leq 3^9+h_9$ (where $h_9$ is defined with $b=3$). The graph of $f_k(x-3^9)$ is colored green. The upper bound from Theorem \ref{Thm4} is the polygonal path colored purple. The points $(G_9(m),H_9(m))$ for $m\in\{2,3,\ldots,9\}$ are colored blue.} \label{Fig3} 
\end{figure} 
\section{Concluding Remarks} 
We wish to acknowledge some of the potential uses and extensions of results derived in this paper. First, we note that Theorem \ref{Thm2} allows us to derive explicit formulas for the values of $s_b(n)$ for integers $n$ whose ordinary base $b$ expansions have certain forms. For example, we were able to invoke Theorem \ref{Thm2} in the proof of Theorem \ref{Thm1} in order to show that $s_b(G_k(m))=H_k(m)$. As another example, it is possible to use Theorem \ref{Thm2} to show that 
\begin{equation} \label{Eq30}
s_b(1+b^{x_1}+b^{x_1+x_2}+b^{x_1+x_2+x_3})=x_1x_2x_3+x_1x_2+x_1x_3+x_2x_3+x_2
-1
\end{equation} for any positive integers $x_1,x_2,x_3$. The equation \eqref{Eq30} appears with several similar identities (many of which can be deduced from Theorem \ref{Thm2}) in \cite{Carlitz64}. 
\par 
Second, arguments based on symmetry and periodicity may be used to extend the upper bounds given by Theorems \ref{Thm3}, \ref{Thm4}, and \ref{Thm5}. For example, referring to the top image in Figure \ref{Fig2}, one will see that the plot of $s_2(n)$ forms several ``mound" shapes. However, only the left sides of the mounds are bounded above by the green curves. It is known \cite[page 2]{Northshield10} that if $k\in\mathbb N$, then  
\begin{equation} \label{Eq24}
s_2(2^k+x)=s_2(2^{k+1}-x)\text{ for all }x\in\{1,2,\ldots,2^k\}.
\end{equation} This allows us to obtain upper bounds over the right sides of the mounds for free. More precisely, since we know from Theorem \ref{Thm5} that $s_2(2^k+x)<f_k(x)$ for all $x\in\{1,2,\ldots,h_k\}$, it follows from \eqref{Eq24} that $s_2(2^{k+1}-x)<f_k(x)$ for all such $x$. Using identities similar to \eqref{Eq24} for arbitrary values of $b$, one may extend our upper bounds for $s_b(n)$ to a larger range of values of $n$.   


\begin{thebibliography}{9}

\bibitem{Berlekamp82}
E. R. Berlekamp, J. H. Conway, and R. K. Guy, Winning ways for your mathematical
plays, Vol. 1: Games in general, Academic Press, London, 1982. 

\bibitem{Calkin09}
N. J. Calkin and H. S. Wilf, Binary partitions of integers and Stern-Brocot-like trees, unpublished
(1998), updated version August 5, 2009, 19 pages.

\bibitem{Carlitz64}
L. Carlitz. A problem in partitions related to the Stirling numbers. \emph{Bull. Amer. Math. Soc.} 70 (1964), 275--278.

\bibitem{Coons14}
M. Coons and J. Tyler. The maximal order of Stern's diatomic sequence. arXiv:1307.1521v2. 

\bibitem{Northshield10}
S. Northshield. Stern’s diatomic sequence $0,1,1,2,1,3,2,3,1,4,\ldots$, \emph{Amer. Math. Monthly}
117 (2010), 581--598.
\end{thebibliography}
\end{document}